\documentclass[11pt]{article}

\usepackage{vmargin}
\setmarginsrb{3cm}{2.5cm}{3cm}{1.5cm}{0cm}{0cm}{0cm}{1.5cm}
\usepackage[utf8]{inputenc}  

\usepackage{enumerate}
\usepackage{dsfont}
\usepackage{amsmath}
\usepackage{mathrsfs}
\usepackage{amsthm}
\usepackage{amsfonts}
\usepackage{amssymb}
\usepackage{mathtools}
\usepackage[toc,page]{appendix}
\usepackage{pgf}
\usepackage{bm}
\usepackage{transparent}
\usepackage{graphicx}
\usepackage{subfigure}
\usepackage{algorithm}
\usepackage{algpseudocode}

\numberwithin{equation}{section}

\newtheorem{thrm}{Theorem}[section]
\newtheorem{prpstn}[thrm]{Proposition}
\newtheorem{dfntn}[thrm]{Definition}
\newtheorem{lmm}[thrm]{Lemma}
\newtheorem{rmrk}[thrm]{Remark}
\newtheorem{assumption}[thrm]{Assumption}

\usepackage[pdftex,pdfborder={0 0 0},
colorlinks=true,
linkcolor=blue,
citecolor=red,
pagebackref=true,
]{hyperref}

\def\<{{\langle}}
\def\>{{\rangle}}
\def\dd{{\rm d}}

\def\mR{\mathbb{R}}

\def\l({\left(}
\def\r){\right)}

\def\T{\mathcal{T}}

\def\dd{\mathrm{d}}
\def \N {\mathbb{N}}
\def\TT{\bar{\T}}

\DeclareMathOperator*{\argmin}{arg\,min}
\DeclareMathOperator{\prox}{prox}
\DeclareMathOperator{\Id}{id}
\DeclareMathOperator{\cvar}{CV@R}

\title{Discrete-time mean field games with risk-averse agents\footnote{This work was supported by a public grant as part of the
Investissement d'avenir project, reference ANR-11-LABX-0056-LMH,
LabEx LMH, and by the FIME Lab (Laboratoire de Finance des Marchés de l'Energie), Paris.} }
\author{J.~Frédéric Bonnans \footnote{CMAP UMR 7641 and Inria, Ecole Polytechnique, route de Saclay, 91128, Palaiseau Cedex, Institut Polytechnique de Paris, France.} \footnote{E-mail: \href{mailto:frederic.bonnans@inria.fr}{frederic.bonnans@inria.fr}}
\and Pierre Lavigne \textsuperscript{$\dagger$}\footnote{E-mail: \href{mailto:pierre.lavigne@polytechnique.edu}{pierre.lavigne@polytechnique.edu}}
\and Laurent Pfeiffer\textsuperscript{$\dagger$}\footnote{E-mail: \href{mailto:laurent.pfeiffer@inria.fr}{laurent.pfeiffer@inria.fr}}}

\begin{document}
\maketitle
\begin{abstract}
We propose and investigate a discrete-time mean field game model involving risk-averse agents. The model under study is a coupled system of dynamic programming equations with a Kolmogorov equation. The agents' risk aversion is modeled by composite risk measures. The existence of a solution to the coupled system is obtained with a fixed point approach. The corresponding 
feedback control allows to construct an approximate Nash equilibrium for a related dynamic game with finitely many players.
\end{abstract}

\section*{Introduction}
The class of mean field games problem was introduced by J-M.~Lasry and P-L.~Lions in \cite{LL06cr1,LL06cr2,LL07mf} and M.~Huang, R.~Malham\'e, and P.~Caines in \cite{HCMieeeAC06}, to study interactions among a large population of players. Many developments and applications have been proposed this last decade, in particular in economics modeling and finance; one can refer for example to Y.~Achdou and al.\@ \cite{ABLLM},  O.~Gu{\'e}ant, J-M.~Lasry and P-L.~Lions \cite{paris-princeton}, and P.~Cardaliaguet and C.-H.~Lehalle  \cite{cardaliaguet2016mean}.
Economic models "\`a la Cournot", considering interactions between the agents via a price variable, have recently received particular attention, let us mention the works of A.~Bensoussan and P.~J.~Graber \cite{graber2015existence}, J.~F.~Bonnans, S.~Hadikanloo, and L.~Pfeiffer \cite{BHP-schauder}, Z. Kobeissi \cite{kobeissi2019classical}, and P.~J.~Graber, V.~Ignazio, and A.~Neufeld \cite{graber2020nonlocal}.

The specificity of the mean field game of this article is the risk aversion of the involved agents. Here risk aversion is modeled with the help of composite risk measures (also called dynamic risk measures). Mathematically, a risk measure $\rho$ is a map that assigns to a random variable $U$ a real number, which is usually high when $U$ is very volatile. In this way $\rho$ can be used to model the reluctance of a player to face highly uncertain expenses.
We refer to the seminal work by P.~Artzner, F.~Delbaen, J-M.~Eber and D.~Heath in \cite{Artzner}.
We will make use of composite risk measures, the natural extension of risk measures to a multistage framework, see for example the article of A.~Shapiro and A.~Ruszczy\'nski \cite{Shapiro-axiom}; for an application to multistage portofolio selection one can refer to A.~Shapiro \cite{shapiro-portofolio}.

Let us describe more precisely our coupled system and the obtained results.
The coupled system describes a population of identical agents which all optimize a linear discrete-time dynamical system (in a continuous state space).
In the model, the associated cost function depends on a variable called belief, which is related to the behavior of the whole group, whence a coupling between a single agent and the population.
Assuming that the population is very large, one can consider that an isolated representative agent has no impact on the belief. Therefore his/her behavior can be conveniently described by dynamic programming equations (in which the belief is a parameter).
Mathematically, the belief is the probability distribution of the states and controls of all agents at the different time steps of the game; it is described via the Kolmogorov equation.
Our first result is an existence result, obtained with a standard fixed point approach. 
In our second result, we show that an optimal feedback control for the mean field game yields an $\varepsilon$-Nash equilibrium for an $N$-player dynamic game, where $\varepsilon \rightarrow 0$ as $N \rightarrow \infty$.
The proof of this result is based on an estimate of the expectation of the Wasserstein distance between the empirical measure of i.i.d.\@ variables and the law of these variables, obtained by N.~Fournier and A.~Guillin \cite[Theorem 1]{Fournier2015}. The approach that we follow was proposed by M.~Huang, P.~Caines, and R.~Malham\'e in \cite{huang2007large}.

Discrete-time and continuous-space mean field game models have been studied in different works. The framework that we propose in this article is close to the one of N.~Saldi, T.~Ba\c sar and M.~Raginsky \cite{basar-discrete}, in particular, we make use of similar weighted spaces. A few works have already investigated the issue of risk aversion. Most of them model risk sensitivity via exponential utility functions, see for example H.~Tembine, Q.~Zhu and T.~Ba\c sar \cite{basar-risk-utility}. The case of robust mean field games is investigated in problem (\textbf{P$2$}) in the work of J.~Moon and T.~Ba\c sar \cite{basar-robust}.
In many economic situations, risk modeling is of interest, in particular in the banking industry \cite{bas2011}. Our approach can also be relevant in situations where mean field games are used to design telecommunication systems or smart grids; see C.~Bertucci et al.\@ \cite{bertucci} and C.~Alasseur, I.~Ben Tahar and A.~Matoussi \cite{Alasseur}. For example, in the latter reference, it could be interesting to take into account the risk of individual no-energy situations or collective black-out situations via robust control.

The article is structured as follows. In Section \ref{pb-formulation} we introduce notations, assumptions, and the system of coupled equations. In Section \ref{int-coupled} we interpret this system as a mean field game system with risk averse agents. In Section \ref{toolbox} we establish general technical results that will be helpful in Section \ref{existence-results}, where we prove the existence of a solution to the coupled system.
Finally in Section \ref{section-6} we investigate the connection between the coupled system and an $N$-player game.

\section{Problem Formulation \label{pb-formulation}}
\subsection{Notations}

We set $\T := \{ 0, \hdots , T-1\}$ and $\TT := \{ 0, \hdots , T\}$ with $T \in \N^\star$. For any $t\in \TT$ and any vector $(x_0,\ldots,x_t)$
we denote
\begin{equation} \nonumber
x_{[t]} := (x_0,\ldots,x_t).
\end{equation}
We denote 
\begin{equation} \nonumber
id : \mathbb{R}^d \to \mathbb{R}^d,
\end{equation} 
the identity mapping.

\subsubsection*{Functions}

Let $C$-Lip denote the set of Lipschitz functions of modulus $C$ on $\mathbb{R}^d$.
We define the $p$-polynomially weighted space
\begin{equation} \nonumber
\mathcal{G}^C_p := \left\{ f \colon \mathbb{R}^d \to \mathbb{R}^{d'} , \; |f(x)| \leq C(|x|^p +1) \right\},
\end{equation}
where the dimension $d'$ depends on the context, with associated norm
\begin{equation} \nonumber
\| f \|_{\mathcal{G},p} := \sup_{x \in \mathbb{R}^d} \frac{|f(x)|}{1+|x|^p}.
\end{equation}
Let $\mathcal{Q}^C_p \subset \mathcal{G}^C_p$ denote the set of convex mappings $f \colon \mathbb{R}^d \to \mathbb{R}$ satisfying
\begin{equation} \label{Q_p}
-C\leq f(x) \leq C(1+ |x|^p), \quad \forall x \in \mathbb{R}^d.
\end{equation}

\subsubsection*{Probability measures}

Let $\mathcal{P}(\mathbb{R}^d)$ denote the set of probability measures on $\mathbb{R}^d$. Given $p \in [1, +\infty)$, we define the set of finite $p$-th order moment measures
\begin{equation} \nonumber
\mathcal{P}_p(\mathbb{R}^d) := \left \lbrace m \in \mathcal{P}(\mathbb{R}^d) , \; \int_{\mathbb{R}^d} |x|^p \dd m(x) < + \infty \right\rbrace,
\end{equation}
that we endow with the Rubinstein-Kantorovitch distance, defined by
\begin{equation} \nonumber
d_1(\mu,\nu) := \sup_{\phi \in 1-\text{Lip}} \int_{\mathbb{R}^d} \phi(x) \dd (\mu - \nu)(x),
\end{equation}
for any $\mu$ and $\nu \in \mathcal{P}_1(\mathbb{R}^d)$ (see \cite[Particular case 5.15]{villani} for more details). We recall that by the H\"older inequality, $\mathcal{P}_p(\mathbb{R}^d) \subseteq \mathcal{P}_1(\mathbb{R}^d)$ for any $p>1$.
Given $C>0$, we define
\begin{equation} \nonumber
\mathcal{P}^C_p(\mathbb{R}^d) := \left \lbrace m \in  \mathcal{P}_p(\mathbb{R}^d) , \, \int_{\mathbb{R}^d} |x|^p \dd m(x) \leq  C \right \rbrace.
\end{equation}
We also consider the following sets of beliefs
\begin{equation} \nonumber
\mathcal{B}_2 := (\mathcal{P}_2(\mathbb{R}^{2d}))^{T} \times \mathcal{P}_2(\mathbb{R}^d), \qquad
\mathcal{B}^C_2 := (\mathcal{P}^C_2(\mathbb{R}^{2d}))^{T} \times \mathcal{P}^C_2(\mathbb{R}^d),
\end{equation}
endowed with the Rubinstein-Kantorovitch distances for the product topology, also denoted $d_1$.

For any $m$ and $\nu \in \mathcal{P}(\mathbb{R}^d)$, we define the convolution product $\nu \ast m$ by
\begin{equation} \label{convolution-measure}
\int_{\mathbb{R}^d} h(x) \dd (\nu \ast m) (x) := \int_{\mathbb{R}^d}\int_{\mathbb{R}^d} h(y + z) \dd \nu(y) \dd m (z),
\end{equation}
for any bounded Borel map $h \in  \mathbb{R}^d \to \mathbb{R}$.
For any $m \in \mathcal{P}(\mathbb{R}^d)$ and for any Borel map $g \colon \mathbb{R}^d \to \mathbb{R}^{d'}$, we define the image measure $g \sharp m \in \mathcal{P}(\mathbb{R}^{d'})$ by
\begin{equation}  \label{push-forward}
\int_{\mathbb{R}^d} (h \circ g) (x) \dd m(x) = \int_{\mathbb{R}^d} h (y) \dd g 
 \sharp m (y),
\end{equation}
for any bounded Borel map $h \in  \mathbb{R}^d \to \mathbb{R}^{d'}$.

\subsection{Coupled system} \label{subsection:coupled_system}

Let us first introduce the data of the problem. We consider
\begin{itemize}
\item a congestion function $F \colon \TT \times \mathbb{R}^d \times  \mathcal{B}_2 \to \mathbb{R}$
\item a price function $P \colon \T \times \mathcal{B}_2 \to \mathbb{R}^d$
\item an initial distribution $\bar{m} \in \mathcal{P}_2(\mathbb{R}^d)$
\item individual noise distributions $(\nu(t))_{t \in \T} \in (\mathcal{P}_2(\mathbb{R}^d))^T$.
\end{itemize}
The running cost $\ell \colon \T \times \mathbb{R}^d \times \mathbb{R}^d \times  \mathcal{B}_2 \to \mathbb{R}$ is defined by
\begin{equation} \nonumber
\ell(t,x,a,b) = \frac{1}{2} |a|^2 + \langle a,P(t,b) \rangle + F(t,x,b).
\end{equation}
For modeling risk aversion, we consider a family of subsets $(\mathcal{Z}_t)_{t \in \T}$ such that
\begin{equation*}
\mathcal{Z}_t \subseteq \left\{Z \in L^{\infty}(\mathbb{R}^d), \, \int_{\mathbb{R}^d}Z(y) \dd \nu(t,y) = 1,\; Z \geq 0 \right\}, \quad \forall t \in \T.
\end{equation*}
For any $t \in \T$, we define
\begin{equation} \label{eq:mt}
 \mathcal{M}_t := \left\{ \xi \in \mathcal{P}(\mathbb{R}^d),\; \dd \xi =  Z \dd \nu(t),\; Z\in \mathcal{Z}_t \right\}.
\end{equation}
For any $t \in \T$, $\mathcal{Z}_t$ is assumed to be nonempty and convex, thus $\mathcal{M}_t$ is a nonempty and convex subset of $\mathcal{P}(\mathbb{R}^d)$.

We propose to study a \textit{risk averse mean field game} (MFG), taking the form of the following coupled system:
\begin{equation} \label{system} \tag{MFG}
\left\{
\begin{array}{cl}
(\text{i}) &
\begin{cases} \displaystyle
u(t,x) =\inf_{a \in \mathbb{R}^d} \left(\ell(t,x,a,b) + \sup_{\xi \in \mathcal{M}_t} \int_{\mathbb{R}^d} u(t+1,x+ a + y)\dd \xi(y) \right),\\
u(T,x) = F(T,x,b),
\end{cases}\\
\\
(\text{ii}) & \displaystyle
\alpha_t(x) = \argmin_{a \in \mathbb{R}^d} \left(\ell(t,x,a,b) + \sup_{\xi \in \mathcal{M}_t} \int_{\mathbb{R}^d} u(t+1,x + a + y)\dd \xi(y) \right),\\
\\
(\text{iii}) &
\begin{cases}
 m(t+1) =  \nu(t) \ast [ (id + \alpha_t ) \sharp m(t) ],\\
 m(0) = \bar{m},
\end{cases}\\
\\
(\text{iv}) &  \mu(t) = (id,\alpha_t) \sharp m(t),\\
\\
(\text{v}) & b := (\mu(0), \dots, \mu(T-1), m(T)),
\end{array}
\right.
\end{equation}
for any $(t,x) \in \T \times \mathbb{R}^d$.
The five unknowns in the above system are
\begin{itemize}
\item the value function $u \in (\mathcal{G}_2)^{T+1}$
\item the feedback control $\alpha \in (\mathcal{G}_1 \cap 1\text{-Lip})^T$
\item the distribution of states $m \in(\mathcal{P}_2)^{T+1}$
\item the joint distribution of states and controls $\mu \in(\mathcal{P}_2(\mathbb{R}^{2d}))^T$
\item the belief $b \in \mathcal{B}_2$.
\end{itemize}

Let us describe briefly the coupled system; we will justify it more in detail in Section \ref{int-coupled}.
Equation (\ref{system},i) is a dynamic programming equation associated with a discrete-time optimal control problem for a representative agent. The belief $b$ appears as a parameter of the equation, since a single agent has no impact on it.
The corresponding optimal feedback control $\alpha$ is then given by (\ref{system},ii).
Now, assuming that all agents make use of the feedback control $\alpha$, the distribution of their state $m$ is described by the Kolmogorov equation (\ref{system},iii) with initial condition $\bar{m}$.

Our approach for proving the existence of a solution consists in formulating the system \eqref{system} as a fixed point equation. For this purpose, we consider two mappings. The first one, that we call dynamic programming mapping, assigns to a belief $b$ the solutions $u^\star(b)$ and $\alpha^\star(b)$ to equations (\ref{system},i) and (\ref{system},ii), respectively. The second one, the Kolmogorov mapping, assigns to a feedback control $\alpha$ the triplet $(m^\star(\alpha),\mu^\star(\alpha),b^\star(\alpha))$, where $m^\star(\alpha)$, $\mu^\star(\alpha)$, and $b^\star(\alpha)$ are the solutions to (\ref{system},iii), (\ref{system},iv), and (\ref{system},v), respectively.
These two mappings will be investigated in Section \ref{existence-results}. They allow to reformulate the system (\ref{system}) as an equivalent fixed point equation
\begin{equation*}
b= b^\star \circ \alpha^\star(b).
\end{equation*}

\subsection{Assumptions}

We state now the assumptions on the data of the problem, in force all along the article. Note that for the results of Section \ref{section-6} (dealing with the $N$-player dynamic game), we will need a slightly stronger assumption on the mapping $F$.

We make use of the same constant $C$ to formulate the different assumptions. In the sequel, the constant $C$ denotes a generic constant depending only on those involved in the assumptions and $T$; its value can change from an inequality to the next one.

\begin{assumption} \label{A}
There exists $C>0$ such that $\bar{m} \in \mathcal{P}_2^C(\mathbb{R}^d)$ and such that for any $t \in \T$, $\nu(t) \in \mathcal{P}_2^C(\mathbb{R}^d)$.
\end{assumption}

\begin{assumption} \label{Z}
There exists $C>0$ such that for any $t\in\T$ and for any $Z \in \mathcal{Z}_t$,
\begin{equation}  \nonumber
\|Z\|_{\infty} \leq C,
\end{equation}
and there exists $Z' \in \mathcal{Z}_t$ such that
\begin{equation}  \nonumber
Z' \geq \frac{1}{C} \quad \text{a.e.}
\end{equation}
\end{assumption}

\begin{rmrk} \label{M}
Assumption \ref{Z} implies the existence of  $C>0$ such that \begin{equation} \label{assum-Mt}
\mathcal{M}_t \subseteq \mathcal{P}_2^{C}(\mathbb{R}^d), \quad \forall t \in \T.
\end{equation}
The results obtained in Section \ref{existence-results} only require \eqref{assum-Mt} to hold. The full Assumption \ref{Z} will be used in Section \ref{section-6}.
\end{rmrk}

\begin{assumption}
\label{B} There exists $C>0$ such that for any $t\in \T$ and for any $b_1$ and $b_2 \in \mathcal{B}_2$,
\begin{equation} \begin{array}{cl}
({\normalfont \text{i}}) & F(t,\cdot, b_1) \in \mathcal{Q}^C_2,\\
 ({\normalfont \text{ii}}) & \|F(t,\cdot, b_1) - F(t,\cdot, b_2)\|_{\mathcal{G},2} \leq C d_1(b_1,b_2),\\
 ({\normalfont \text{iii}}) & |P(t,b_1) - P(t,b_2)| \leq C d_1(b_1,b_2),\\
 ({\normalfont \text{iv}}) &  |P(t,b_1)| \leq C.
\end{array} \nonumber
\end{equation}
\end{assumption}

\begin{rmrk}
In economics or in finance, prices typically depend on the aggregated demand or supply. One could consider for example
\begin{equation} \nonumber
         P(t,b) := \psi\left(t, \int_{\mathbb{R}^{2d}} \alpha \dd \mu(t,x,\alpha)  \right),
    \end{equation}
where $\psi \colon \T \times \mathbb{R}^d \to \mathbb{R}^d$. In this case, if $\psi$ is a $C$-Lipschitz mapping then for any $b_1$ and $b_2 \in \mathcal{B}_2$, one has that
\begin{equation} \nonumber
|P(t,b_1) - P(t,b_2)|  \leq C \left|\int_{\mathbb{R}^{2d}} \alpha \dd (\mu_1 -\mu_2) (t,x,\alpha) \right|
 \leq C d_1(\mu_1,\mu_2) \leq C d_1(b_1,b_2),
\end{equation}
which implies Assumption \ref{B} {\normalfont(iii)}.  Assumption \ref{B} {\normalfont(iv)} also holds if $|\psi| \leq C$.
\end{rmrk}

\section{Interpretation of the coupled system \label{int-coupled}}

In Subsection \ref{3.1} we describe the risk averse optimal control problem associated with (\ref{system},i-ii). In Subsection \ref{3.2} we justify the Kolmogorov equation (\ref{system},iii).

\subsection{Dynamic programming equation \label{3.1}}

\subsubsection*{Risk measures}

Let $X_0$ and $(Y_{t})_{t\in \T}$ be $(T+1)$-independent random variables defined on a probability space $(\Omega, \mathcal{F}, \mathbb{P})$. Let $\mathcal{L}(X_0) = \bar{m}$ and  $\mathcal{L}(Y_t) = \nu(t)$. We define the filtration $(\mathcal{F}_t)_{t\in \T}$, where $\mathcal{F}_0 := \sigma(X_0)$ is the sigma-algebra generated by $X_0$, and $\mathcal{F}_{t+1} := \sigma(X_0,Y_{[t]})$.
We denote for any $t \in \TT$ and any $p\in [1,+\infty)$
\begin{equation} \nonumber
 \mathbb{L}_t^p(\Omega,\mathbb{R}^{d'}) := L^p(\Omega,\mathcal{F}_t, \mathbb{P},\mathbb{R}^{d'}),
\end{equation}
the space of $\mathcal{F}_t$ measurable random variables with finite $p$-th order moment and value in $\mathbb{R}^{d'}$. When the dimension is $d'=1$, we simplify the notation: $\mathbb{L}_t^p := \mathbb{L}_t^p(\Omega,\mathbb{R})$.

\begin{dfntn}
Given $t \in \T$, we say that a mapping $\rho_{t} \colon  \mathbb{L}_{t+1}^1 \to \mathbb{L}_{t}^1$ is a {\em one-step conditional
risk mapping} if it satisfies the following conditions:
\begin{itemize}
\item \textbf{(M) Monotonicity:} For any $U$  and $U' \in \mathbb{L}_{t+1}^1$ such that $U \leq U'$, we have
\begin{equation*}
\rho_{t} (U) \leq \rho_{t} (U'), \quad \text{a.s.}
\end{equation*}
\item \textbf{(C) Convexity:} For any $U$ and $U' \in \mathbb{L}_{t+1}^1$, for any $\alpha \in [0,1]$, we have 
\begin{equation*}
\rho_{t} (\alpha U + (1-\alpha)U' ) \leq \alpha \rho_{t} (U) + (1-\alpha) \rho_{t}(U'), \quad \text{ a.s.}
\end{equation*}
\item \textbf{(TI) Translation Invariance:} For any $U \in \mathbb{L}_{t+1}^1$ and for any $V \in \mathbb{L}_{t}^1$, we have
\begin{equation*}
\rho_{t} (U + V ) = \rho_{t} (U ) + V, \quad \text{a.s.}
\end{equation*}
\item \textbf{(PH) Positive Homogeneity:} For any $\alpha \geq 0$, for any $U \in \mathbb{L}_{t+1}^1$, we have
\begin{equation*}
\rho_{t} (\alpha U ) = \alpha\rho_{t} (U), \quad \text{a.s.}
\end{equation*}
\end{itemize}
\end{dfntn}
Quoting \cite{Rus-risk-2010}, the condional risk mapping $\rho_t(U_{t+1})$ can be interpreted as a fair one-time $\mathcal{F}_t$-measurable charge we would be willing to incur at time $t$ instead of the random futur cost $U_{t+1}$.

We fix now a family of {\em one-step conditional
risk mapping}  $(\rho_{t})_{t\in \T}$, $\rho_{t} \colon \mathbb{L}_{t+1}^1 \to \mathbb{L}_{t}^1$, defined by
\begin{equation} \label{riskmeasure}
\rho_{t}(U_{t+1})(x_0,y_{[t-1]}) =
\sup_{Z \in \mathcal{Z}_t} \int_{\Omega} U_{t+1}(x_0,y_{[t-1]},Y_t(\omega)) Z(Y_t(\omega)) \dd \mathbb{P}(\omega),
\end{equation}
where the random variables $U_{t+1}$ and $\rho_t(U_{t+1})$ are explicitly represented as measurable functions of $(x_0,y_{[t]}) \in \mathbb{R}^{(t+2)d}$ and $(x_0,y_{[t-1]}) \in \mathbb{R}^{(t+1)d}$, respectively. Recalling the definition of $\mathcal{M}_t$ \eqref{eq:mt}, we have 
\begin{equation} \nonumber
\rho_{t}(U_{t+1})(x_0,y_{[t-1]}) =
\sup_{\xi \in \mathcal{M}_t} \int_{\mathbb{R}^d} U_{t+1}(x_0,y_{[t-1]},y_t)  \dd \xi(y_t).
\end{equation}
We set
\begin{equation} \nonumber
\mathcal{Q}_{t+1} := \left\{Q = Z(Y_{t}) \text{ a.s.}, \, Z \in \mathcal{Z}_t\right\}
\end{equation}
so that $\rho_t$ can be expressed in the following form:
\begin{equation} \nonumber
\rho_{t}(U_{t+1})  = \sup_{Q_{t+1} \in \mathcal{Q}_{t+1}} \mathbb{E} \left[U_{t+1} Q_{t+1} \vert \mathcal{F}_t \right].
\end{equation}

Finally we construct the associated composite risk measure $\rho \colon  \mathbb{L}_{T}^1  \to \mathbb{R}$,
\begin{equation} \nonumber
\rho(U) := \mathbb{E} \left[ \rho_{0}  \circ \cdots  \circ  \rho_{ T-1}(U) \right],
\end{equation}
which also satisfies (\textbf{M}), (\textbf{C}), (\textbf{TI}), and (\textbf{PH}).

\begin{rmrk}
Given a probability space $(\Omega',\mathcal{F}',\mathbb{P}')$ and given $\alpha \in (0,1]$, the conditional value at risk (also called expected shortfall or average value at risk) of a random variable $U \in L^1(\Omega',\mathcal{F}',\mathbb{P}')$ is defined by
\begin{equation*}
\cvar_{\alpha}(U):= \inf_{W \in L^1(\Omega',\mathcal{F}',\mathbb{P}')}
W + \alpha^{-1} \mathbb{E} \left[ (U-W)_{+} \right],
\end{equation*}
where $x_+= \max \{ 0, x \}$ denotes the positive part of any $x \in \mathbb{R}$. It has the following dual representation (see \cite[Lemma 4.51 and Theorem 4.52]{fol11}):
\begin{equation*}
\cvar_{\alpha}(U)= \sup \left\{
\mathbb{E} \left[ UZ \right] \,\Big|\,
Z \in L^\infty(\Omega',\mathbb{F}',\mathbb{P}'),\,
Z \in [0,\alpha^{-1}] \text{ a.s.},\,
\mathbb{E}[Z]= 1
\right\}.
\end{equation*}
Therefore, a natural extension of the conditional value at risk to the framework of the article is given by
\begin{equation} \nonumber
\rho_t(U_{t+1}) = \sup_{Z \in \mathcal{Z}_{t}} \mathbb{E}\left[ U_{t+1} Z(Y_t)  \vert \mathcal{F}_t \right],
\end{equation}
where
\begin{equation} \nonumber
\mathcal{Z}_{t} := \left\{ Z \in L^{\infty}(\mathbb{R}^d) \,\Big|\,
Z \in [0,\alpha^{-1}] \text{ a.e.},\,
\int_{\mathbb{R}^d} Z(y) \dd \nu(t,y) = 1 \right\}.
\end{equation}
This particular definition of $\mathcal{Z}_t$ satisfies Assumption \ref{Z}.
We refer to \cite[Definition 11.8]{fol11} and \cite[Subsection 2.3.1]{che11} for extensions of the conditional value at risk to general filtrations in a discrete-time setting.
\end{rmrk}

\begin{rmrk}
The risk measure that we have constructed does not have the most general structure possible.
In our setting, the sets $\mathcal{M}_t$ are fixed. In \cite{Rus-risk-2010}, these sets depend on the current state and control (see in particular Sections 4 and 5). In this more general context, it is still possible to derive a dynamic programming principle for the underlying optimal control problem (see \cite[Theorem 2]{Rus-risk-2010}).
However, the convexity of the value function, which plays an important role in our analysis, is lost in such a setting.
\end{rmrk}

\subsubsection*{Control problem}

We consider the following set of controls for any $t\in \T$,
\begin{equation} \nonumber
\mathcal{A}_t = \mathbb{L}^2_t(\Omega,\mathbb{R}^d), \qquad \mathcal{A} := \mathcal{A}_0 \times \cdots \times \mathcal{A}_{T-1}.
\end{equation}
Given a control $A \in \mathcal{A}$, the evolution of the state of the representative player is given by
\begin{equation} \label{state-eq} \tag{\textit{C}}
X_{t+1} = X_t + A_t + Y_t, \quad \forall t \in \T.
\end{equation}
The initial condition is the random variable $X_0$ fixed previously.
Will call the the variable $(X_t)_{t \in \TT}$ associated state with $A$.
In the notation, we do not make explicit the dependence of $(X_t)_{t \in \TT}$ with respect to $A$, which is always clear from the context. Note that by induction, $X_t \in \mathbb{L}^2_{t}(\Omega,\mathbb{R}^d)$ for any $t \in \TT$.

For a given belief $b \in \mathcal{B}_2$, the risk averse multistage cost of the representative agent is given by
\begin{equation} \label{cost-J}
\mathcal{J}(A,b) :=   \rho \left( \sum_{t = 0}^{T-1}  \ell (t, X_t , A_t, b)  +F(T,X_T, b) \right).
\end{equation}
The corresponding problem is
\begin{equation} \label{risk-pb} \tag{\textit{P}}
\inf_{A \in \mathcal{A}} \mathcal{J}(A,b).
\end{equation}

In what follows, we show how equations (\ref{system},i) and (\ref{system},ii) allow to characterize the unique solution to \eqref{risk-pb}.
Let us recall that $b$ is fixed in this subsection.
Let us denote by $u \in (\mathcal{G}_2)^{T+1}$ the solution to (\ref{system},i) and let us denote by $\alpha \in (\mathcal{G}_1 \cap 1\text{-Lip})^T$ the solution to (\ref{system},ii).
The existence and uniqueness of these solutions will be independently established in  Lemma \ref{u-alpha-prox-lemma} and Lemma \ref{alpha-u-in-G1-G2}.

\begin{lmm} \label{lemma:def_bar_a}
There exists a unique control $\bar{A} \in \mathcal{A}$ with associated state $\bar{X}$ such that for all $t \in \mathcal{T}$,
\begin{equation} \label{open-loop}
\bar{A}_t = \alpha_t(\bar{X}_t), \quad \text{a.s.}
\end{equation}
\end{lmm}

\begin{proof}
Let $(\bar{X}_t)_{t \in \TT}$ be the solution to the closed-loop system
\begin{equation}  \label{closed-loop}
\bar{X}_{t+1}= \bar{X}_t + \alpha_t(\bar{X}_t) + Y_t, \quad \forall t \in \mathcal{T}.
\end{equation}
It is easy to verify by induction that for all $t \in \bar{\mathcal{T}}$, the random variable $\bar{X}_t$ is $\mathcal{F}_t$-measurable and has a bounded second-order moment. Indeed, $\alpha_t$ is Lipschitz-continuous, thus has a linear growth; therefore, if $\bar{X}_t$ has a bounded second-order moment, then $\alpha_t(\bar{X}_t)$ also has a bounded second-order moment.
We define now $\bar{A}$ by
\begin{equation} \label{open-loop2}
\bar{A}_t = \alpha_t(\bar{X}_t).
\end{equation}
Since $\bar{X}_t$ is adapted to $\mathcal{F}_t$, we also have that $\bar{A}_t$ is $\mathcal{F}_t$-measurable. As we already pointed out, $\alpha_t(\bar{X}_t)$ has a bounded second-order moment.
This proves that $\bar{A}\in \mathcal{A}$.
Finally, it is clear that by \eqref{closed-loop} and \eqref{open-loop2}, the pair $(\bar{A},\bar{X})$ satisfies the state equation \eqref{state-eq}.

Let us justify the uniqueness of $\bar{A}$. Let $\tilde{A} \in \mathcal{A}$ be such that $\tilde{A}_t= \alpha_t(\tilde{X}_t)$, where $\tilde{X}$ is the associated state. Then, $\tilde{X}$ is a solution to the closed-loop system \eqref{closed-loop}. Therefore $\tilde{X}= \bar{X}$ and finally $\tilde{A}_t= \alpha_t(\tilde{X}_t)= \alpha_t(\bar{X}_t)= \bar{A}_t$.
The lemma is proved.
\end{proof}

The following proposition states the optimality of the control $\bar{A}$.

\begin{prpstn} \label{proposition:dyn_prog}
We have
\begin{equation} \label{eq:value_opt}
\inf_{A \in \mathcal{A}} \mathcal{J}(A,b)
=  \mathbb{E}\left[ u(0,X_0) \right] = \int_{\mathbb{R}^d}u(0,x) \dd m(0,x),
\end{equation}
where $u$ solves the dynamic programming equation {\normalfont(\ref{system},i)}. Moreover, the control $\bar{A}$ defined in Lemma \ref{lemma:def_bar_a} is the unique solution to Problem \eqref{risk-pb}.
\end{prpstn}

\begin{proof}
The proof is directly adapted from \cite[Theorem 2]{Rus-risk-2010}.
As a consequence of the translation invariance property (\textbf{TI}), the problem (\ref{risk-pb}) can be expressed in a nested form
\begin{align}
\inf_{A \in \mathcal{A}} \mathcal{J}(A,b)
= & \mathbb{E}\bigg[ \inf_{A_{0} \in \mathcal{A}_0} \ell (0, X_0 , A_0, b) +  \rho_{0} \bigg( \inf_{A_{1} \in \mathcal{A}_1}  \ell (1, X_1, A_1, b) + \cdots   
\notag \\
& \qquad +  \rho_{T-2}  \bigg( \inf_{A_{T-1}\in \mathcal{A}_{T-1}} \ell (T-1,X_{T-1}, A_{T-1}, b) + \rho_{T-1} \bigg( 
F(T,X_T, b)  \bigg)\bigg) \cdots \bigg)\bigg]. \label{risk-pb-nested}
\end{align}
By (\ref{system},i), we have $u(T,X_T)= F(T,X_T,b)$ almost surely.
We also have $X_T= X_{T-1} + A_{T-1} + Y_{T-1}$, as a consequence of the state equation \eqref{state-eq}.
Therefore, the innermost subproblem in \eqref{risk-pb-nested} is given by
\begin{equation} \label{eq:subproblem}
\inf_{A_{T-1}\in \mathcal{A}_{T-1}}   \ell (T-1,X_{T-1}, A_{T-1}, b) + \rho_{T-1} (u(T,X_{T-1} + A_{T-1} + Y_{T-1})).
\end{equation}
Since $X_{T-1}, A_{T-1} \in \mathcal{F}_{T-1}$, the unique solution to subproblem  \eqref{eq:subproblem} 
is $A_{T-1}= \alpha_{T-1}(X_{T-1})$.
Moreover, the value of subproblem \eqref{eq:subproblem} is $u(T-1,X_{T-1})$.
Proceeding iteratively for all times $t \in \mathcal{T}$, we conclude that \eqref{eq:value_opt} holds and that any solution $A$ to problem \eqref{risk-pb} with associated state $X$ satisfies $A_t= \alpha_t(X_t)$. Therefore, by Lemma \ref{lemma:def_bar_a}, $\bar{A}$ is the unique solution to \eqref{risk-pb}. The proof is complete.
\end{proof}

\subsection{Kolmogorov equation \label{3.2}}
\begin{lmm}
Let $\alpha \colon \T \times \mR^d \to \mR^d$ be a continuous vector field. Suppose that the state equation \eqref{state-eq} is of the feedback form
\begin{equation} \nonumber
X_{t+1} = X_t + \alpha_t(X_t) + Y_t.
\end{equation}
 Then for any $t\in \TT$, $m(t) = \mathcal{L}(X_t) \in \mathcal{P}(\mathbb{R}^d)$ is characterized by the Kolmogorov equation {\normalfont(\ref{system},\text{iv})}.
\end{lmm}

\begin{proof}
Let $\phi$ be a bounded Borel test function. For any $t\in \T$, by independence of $X_t$ and $Y_t$ we have
\begin{align} \nonumber
\mathbb{E} \left[ \phi \left( X_{t+1}\right)  \right] & = \mathbb{E} \left[ \phi \left(X_t + \alpha_t(X_t) + Y_t \right) \right]\\
& =\int_{\mathbb{R}^d}  \int_{\mathbb{R}^d} \phi(x + \alpha_t(x) + y ) \dd m(t,x) \dd \nu(t,y). \nonumber
\end{align}
By definition of the push-forward \eqref{push-forward} we obtain
\begin{equation} \nonumber
 \int_{\mathbb{R}^d} \phi(x + \alpha_t(x) + y ) \dd m(t,x)
 =  \int_{\mathbb{R}^d} \phi( z + y ) \dd (id + \alpha_t)\sharp m(t,z).
 \end{equation}
By definition of convolution \eqref{convolution-measure} we have
\begin{align} \nonumber
\int_{\mathbb{R}^d} \int_{\mathbb{R}^d} \phi(z + y ) \dd \nu(t,y) \dd (id + \alpha_t)\sharp m(t,z) =  \int_{\mathbb{R}^d} \phi(x)  \dd \left(\nu(t) \ast \left[(id + \alpha_t)\sharp m(t) \right]\right)(x),
\end{align}
as was to be proved.
\end{proof}

\section{Technical lemmas \label{toolbox}}

This section contains independent technical lemmas. The reader only interested in the main results of the article can skip it.

\begin{lmm} \label{convol-p1}
Let $p\in [1,+\infty)$ and let $C>0$. For any $m_1$ and $m_2$ in $\mathcal{P}^C_p(\mathbb{R}^d)$, the probability measure $m_1 \ast m_2$ lies in $\mathcal{P}^{2^{p}C}_p(\mathbb{R}^d)$. In addition, given $m_0 \in \mathcal{P}_p^C(\mathbb{R}^d)$, the mapping $\mathcal{P}^C_p(\mathbb{R}^d) \ni m \mapsto m_0 \ast m $ is non-expansive for the distance $d_1$.
\end{lmm}

\begin{proof}
Let $m_1$ and $m_2$ in $\mathcal{P}_p^C(\mathbb{R}^d)$. We have
\begin{align*}
\int_{\mathbb{R}^d} |x|^p \dd (m_1 \ast m_2) (x) & = \int_{\mathbb{R}^d} \int_{\mathbb{R}^d}  |y+z|^p \dd m_1(y) \dd m_2(z) \\
& \leq  \int_{\mathbb{R}^d} \int_{\mathbb{R}^d}   2^{p-1}(|y|^p+|z|^p) \dd m_1(y)\dd m_2(z) \leq 2^{p}C.
\end{align*}
Thus $m_1 \ast m_2 \in \mathcal{P}^{2^{p}C}_p(\mathbb{R}^d)$. Moreover, given $m_0 \in \mathcal{P}_p^C(\mathbb{R}^d)$, we have
\begin{align*}
 d_1(m_0 \ast m_1, m_0 \ast m_2) & = \sup_{\phi \in 1\mathrm{-Lip}} \int_{\mathbb{R}^d} \phi(x) \dd (m_0 \ast m_1 - m_0 \ast m_2)(x) \\
 & = \sup_{\phi \in 1\mathrm{-Lip}} \int_{\mathbb{R}^d}\left(\int_{\mathbb{R}^d} \phi(y+z)\dd m_0(y) \right) \dd (m_1 - m_2)(z).
\end{align*}
Since the mapping $z \mapsto \int_{\mathbb{R}^d} \phi(y+z)\dd m_0(y)$ is non-expansive, we further obtain that
\begin{equation} \nonumber
d_1(m_0 \ast m_1, m_0 \ast m_2)  \leq d_1(m_1,m_2),
\end{equation}
which concludes the proof.
\end{proof}

\begin{lmm} \label{g-sharp-m}
Let $p\in [1,+\infty)$ and let $C>0$.  For any $m \in \mathcal{P}^C_p(\mathbb{R}^d)$ and for any Borel map $g \in \mathcal{G}_1^C$, the probability measure $g \sharp m$ lies in $\mathcal{P}^{q}_p(\mathbb{R}^d)$, with $q = 2^{p-1} C^p(1 + C)$. In addition, the inequality
\begin{equation} \label{d1mg}
d_1(g_1 \sharp m_1,g_2 \sharp m_2) \leq (1+C)\|g_1-g_2\|_{\mathcal{G},1} + Cd_1(m_1,m_2)
\end{equation}
holds for any $m_1$ and $m_2$ in $\mathcal{P}^C_p(\mathbb{R}^d)$ and for any Borel maps $g_1$ and $g_2$ in $\mathcal{G}_1^C \cap C\mathrm{-Lip}$.
\end{lmm}

\begin{proof}
Let $m \in \mathcal{P}^C_p(\mathbb{R}^d)$ and let $g \in \mathcal{G}_1^C$ be a Borel map. By \eqref{push-forward} we have
\begin{equation} \nonumber
\int_{\mathbb{R}^d} |x|^p \dd g \sharp m (x) = \int_{\mathbb{R}^d} |g(x)|^p \dd m (x) 
 \leq \| g \|_{\mathcal{G},1}^p \int_{\mathbb{R}^d} (1+|x|)^p \dd m (x) \leq q.
\end{equation}
Consider $(g_1,m_1)$ and $(g_2,m_2)$ in $\mathcal{G}_1^C \times \mathcal{P}^C_p(\mathbb{R}^d)$. We have
\begin{align*}
& d_1(g_1 \sharp m_1,g_2 \sharp m_2) = \sup_{\phi \in 1\mathrm{-Lip}}\int_{\mathbb{R}^d} \phi(x) \dd (g_1 \sharp m_1 - g_2 \sharp m_2)(x) \\
& \qquad = \sup_{\phi \in 1\mathrm{-Lip}}\int_{\mathbb{R}^d} (\phi \circ g_1(x) - \phi \circ g_2(x)) \dd m_2(x) + \int_{\mathbb{R}^d} \phi \circ g_1(x) \dd (m_1-m_2)(x) \\
& \qquad \leq \|g_1-g_2\|_{\mathcal{G},1} \int_{\mathbb{R}^d}(1+|x|)\dd m_2(x) +  \sup_{\phi \in 1\mathrm{-Lip}} \int_{\mathbb{R}^d}  \phi \circ g_1(x) \dd (m_1-m_2)(x)\\
& \qquad \leq (1+C) \|g_1-g_2\|_{\mathcal{G},1} + C \sup_{\phi \in 1\mathrm{-Lip}} \int_{\mathbb{R}^d} C^{-1} \phi \circ g_1(x)  \dd (m_1-m_2)(x). \label{continuous-g-m}
\end{align*}
Observing that $C^{-1} \phi \circ g_1 \in 1\mathrm{-Lip}$, we deduce inequality (\ref{d1mg}).
\end{proof}

Given a convex function $u \colon \mathbb{R}^d \to \mathbb{R}$, we define the Moreau envelope $V_u$ and the proximal operator $\prox_u$ of $u$ as follows:
\begin{equation} \label{prox-operator}
V_u(x) := \min_{y \in \mathbb{R}^d} \frac{1}{2}|x-y|^2 + u(y), \qquad \prox_u(x) := \argmin_{y \in \mathbb{R}^d} \frac{1}{2}|x-y|^2 + u(y).
\end{equation}
In the proofs, we will occasionally consider the map $g_u \colon \mathbb{R}^d \times \mathbb{R}^d \rightarrow \mathbb{R}$, defined by
\begin{equation}\nonumber
g_u(x,y) := \frac{1}{2}|x-y|^2 + u(y).
\end{equation}

\begin{prpstn} \label{non-expensive}
Let $u \colon \mathbb{R}^d \to \mathbb{R}$ be a convex function. Then $\prox_u$ and $(\Id - \prox_u)$ are non-expansive.
\end{prpstn}

\begin{proof}
Direct consequence of \cite[Proposition 12.27]{combettes}.
\end{proof}

\begin{lmm} \label{prox-in-G1}
Let $R>0$ and let $u \in \mathcal{Q}_2^R$ (the set was defined in (\ref{Q_p})). Then  $|\prox_u|^2 \in \mathcal{G}_2^{C_1(R)}$ and $|\prox_u| \in \mathcal{G}_1^{(C_1(R))^{1/2}}$, where $C_1(R):=8R+2$.
\end{lmm}

\begin{proof}
Let $u \in \mathcal{Q}_2^R$.
By Proposition \ref{non-expensive}, the map $\prox_u$ is non-expansive. Thus
\begin{equation} \label{proxX0}
|\prox_u(x)| \leq |\prox_u(0)| + |x|.
\end{equation}
In addition, from the definition of the proximal operator (\ref{prox-operator}), we have
\begin{equation*}
\frac{1}{2}|\prox_u(0)|^2 + u(\prox_u(0)) \leq u(0).
\end{equation*}
Since $u \in  \mathcal{Q}_2^R$, we deduce that $|\prox_u(0)|^2 \leq 4R$.
We further obtain with (\ref{proxX0}) that
\begin{equation} \label{prox0M1}
|\prox_u(x)|^2 \leq 2(|x|^2 + |\prox_u(0)|^2)
\leq C_1(R)(1 + |x|^2),
\end{equation}
as was to be proved. 
Taking the square root of \eqref{prox0M1}, we infer that $|\prox_u| \in \mathcal{G}_1^{C_1(R)^{1/2}}$.
\end{proof}

\begin{lmm} \label{V-in-Q2}
Let $R>0$ and let $u \in \mathcal{Q}_2^R$. Then  $V_u \in \mathcal{Q}_2^{C_2(R)}$, where 
\begin{equation} \nonumber
C_2(R):=(R+1)(1 + C_1(R)).
\end{equation}
\end{lmm}

\begin{proof}
Let $u \in \mathcal{Q}_2^R$. Clearly $V_u$ is convex as the infimum with respect to $y\in \mathbb{R}^d$ of the jointly convex map $(x,y) \mapsto g_u(x,y)$.
For any $x \in \mathbb{R}^d$, we have
\begin{equation}\nonumber
V_u(x) = \frac{1}{2}|x-\prox_u(x)|^2 + u(\prox_u(x)),
\end{equation}
by definition of $V_u$ and $\prox_u$.
Since $u \in \mathcal{Q}_2^R$, we further obtain that
\begin{equation}\nonumber
- R \leq V_u(x)  \leq |x|^2 + |\prox_u(x)|^2 + R(1+  |\prox_u(x)|^2).\label{vu-x}
\end{equation}
Applying Lemma \ref{prox-in-G1}, we finally obtain that $V_u \in \mathcal{Q}_2^{C_2(R)}$.
\end{proof}

\begin{lmm} \label{QinH}
Let $R>0$. For any $u$ and $v$ in $\mathcal{Q}_2^{R}$, the inequality
\begin{equation} \label{C3}
\| \prox_{u} - \prox_{v} \|_{\mathcal{G},1} \leq C_3(R)  \|u - v\|_{\mathcal{G},2}^{1/2}
\end{equation}
holds, where $C_3(R) := \sqrt{2(1+C_1(R))}$.
\end{lmm}

\begin{proof}
Let $u$ and $v$ in $\mathcal{Q}_2^R$.
Observing that $g_u$ and $g_w$ are $1$-strongly convex with respect to their second argument, we have
\begin{align*}
\frac{1}{2} | \prox_{u}(x) - \prox_{v}(x)|^2 & \leq g_u(x,\prox_{v}(x)) - g_u(x,\prox_{u}(x)), \\
\frac{1}{2} | \prox_{u}(x) - \prox_{v}(x)|^2 & \leq g_v(x,\prox_{u}(x)) - g_v(x,\prox_{v}(x)).
\end{align*}
Summing up the two inequalities, we obtain that
\begin{align}\nonumber
| \prox_{u}(x) - \prox_{v}(x)|^2 & \leq v(\prox_{u}(x))  - u(\prox_{u}(x)) +  u(\prox_{v}(x))- v(\prox_{v}(x))\\ 
& \leq (2 + |\prox_{u}(x)|^2 + |\prox_{v}(x)|^2) \|u - v\|_{\mathcal{G},2}. \label{eq:prox_ineq_1}
\end{align}
By Lemma \ref{prox-in-G1},
\begin{equation} \label{eq:prox_ineq_2}
2 + | \prox_u(x)|^2 + | \prox_v(x)|^2
\leq 2 + 2C_1(R) (1 + |x|^2)
\leq C_3(R)^2 (1 + |x|^2).
\end{equation}
Combining \eqref{eq:prox_ineq_1} and \eqref{eq:prox_ineq_2} and taking the square root, we obtain \eqref{C3}.
\end{proof}

\begin{lmm} \label{Vu-Vv}
Let $R>0$. For any $u$ and $v$ in $\mathcal{Q}_2^{R}$, we have
\begin{equation} \label{C4}
\| V_u - V_v \|_{\mathcal{G},2} \leq C_4(R)\| u - v \|_{\mathcal{G},2},
\end{equation}
where $C_4(R) := 1+C_1(R)$.
\end{lmm}

\begin{proof}
Let $u$ and $v$ in $\mathcal{Q}_2^{R}$. Recalling the definitions of $g_u$ and $g_v$, we have
\begin{equation*}
V_u(x)-V_v(x) \leq g_u(x,\prox_v(x)) - g_v(x,\prox_v(x))
= u(\prox_v(x)) - v(\prox_v(x)).
\end{equation*}
Lemma \ref{prox-in-G1} yields
\begin{equation*} \nonumber
V_u(x) - V_v(x)
\leq  (1+ \|\prox_v(x) \|_{\mathcal{G},1}^2) \| u - v \|_{\mathcal{G},2}
\leq (1 + C_1(R) (1 + |x|^2)) \| u - v \|_{\mathcal{G},2}.
\end{equation*}
Exchanging $u$ and $v$, we deduce \eqref{C4}.
\end{proof}

\begin{lmm} \label{Vu-xy}
Let $R>0$. For any $u \in \mathcal{Q}_2^{R}$ and for any $(x,y) \in \mathbb{R}^d \times \mathbb{R}^d$,
\begin{equation}
|V_u(x) - V_u(y)| \leq C_5(R)(1 + |x| + |y|)|x-y|, \label{vux-vuy}
\end{equation}
where $C_5(R):= 1+ \sqrt{C_1(R)}$.
\end{lmm}

\begin{proof}
Let $u \in\mathcal{Q}_2^{R}$. We have
\begin{align*}
V_u(x)-V_u(y) \leq \ & g_u(x,\prox_u(y)) - g_u(y,\prox_u(y)) \\
= \ & \frac{1}{2} |x-\prox_u(y)|^2 - \frac{1}{2} |y-\prox_u(y)|^2 \\
\leq \ & \frac{1}{2} |x+y-2 \prox_u(y)| \cdot |x-y|. 
\end{align*}
We further obtain with Lemma \ref{prox-in-G1} that
\begin{equation*}
|x+y-2 \prox_u(y)|
\leq |x| + |y | + 2 \sqrt{C_1(R)} (1 + |y|)
\leq 2 (1 + \sqrt{C_1(R)}) (1 + |x| + |y|).
\end{equation*}
Combining the two obtained inequalities and exchanging $x$ and $y$, we obtain \eqref{vux-vuy}.
\end{proof}

\begin{lmm} \label{Ups}
Let $R> 0$ and let $\mathcal{M}$ be a subset of $\mathcal{P}_2^R(\mathbb{R}^d)$. Given $u \in \mathcal{Q}_2^R$, consider the mapping $\Upsilon[u](x)$ defined for any $x \in \mathbb{R}^d$ by
\begin{equation} \nonumber
 \Upsilon [u](x) := \sup_{\xi \in \mathcal{M}}\int_{\mathbb{R}^d}u(x +y) \dd \xi(y).
\end{equation}
Then $\Upsilon[u] \in \mathcal{Q}_2^{C_6(R)}$, where $C_6(R)= 2R(1+R)$. Moreover, the map $\mathcal{Q}_2^R \ni u \mapsto \Upsilon[u]$ is Lipschitz continuous with modulus $2(1+R)$.
\end{lmm}

\begin{proof}
Let $u \in \mathcal{Q}_2^R$. For any $\xi \in \mathcal{M}$, the map $\mathbb{R}^d \ni x \mapsto \int_{\mathbb{R}^d} u(x+y) \dd \xi(y)$ is convex, as can be easily verified. Thus $\Upsilon[u](x)$ is convex with respect to $x$, as a supremum of convex maps. Moreover, for any $x \in \mathbb{R}^d$, we have
\begin{align} \nonumber
-R \leq \Upsilon[u](x) \leq \sup_{\xi \in \mathcal{M}} \int_{\mathbb{R}^d} 2R(1 + |x|^2 + |y|^2) \dd \xi(y)
\leq  2R(1 + |x|^2 + R). \nonumber
\end{align}
This proves that $\Upsilon[u] \in \mathcal{Q}_2^{C_6(R)}$. Consider now $v \in \mathcal{Q}_2^R$. We have
\begin{align}
|\Upsilon[u](x) - \Upsilon[v](x)| & \leq \sup_{\xi \in \mathcal{M}} \left \vert \int_{\mathbb{R}^d}(u(x +y) - v(x +y)) \dd \xi(y) \right \vert \notag \\
& \leq 
\|u-v\|_{\mathcal{G},2} \left( \sup_{\xi \in \mathcal{M}}  \int_{\mathbb{R}^d}(1+|x+y|^2) \dd \xi(y) \right).
\label{eq:lipsch_ups_1}
\end{align}
For any $\xi \in \mathcal{M}$, we further have
\begin{equation}
\int_{\mathbb{R}^d}(1+|x+y|^2) \dd \xi(y)
\leq 1 + 2 |x|^2 + 2 R
\leq 2(1+R) (1+ |x|^2). \label{eq:lipsch_ups_2}
\end{equation}
Combining \eqref{eq:lipsch_ups_1} and \eqref{eq:lipsch_ups_2}, we deduce that
\begin{equation*}
\| \Upsilon[u] - \Upsilon[v] \|_{\mathcal{G},2}
\leq 2(1+R) \| u - v \|_{\mathcal{G},2},
\end{equation*}
as was to be proved.
\end{proof}

\section{Existence result \label{existence-results}}

In this section we prove the main existence result.
We first investigate the continuity of the dynamic programming mapping and the continuity of the Kolmogorov mapping introduced in Subsection \ref{subsection:coupled_system}.

\subsection{Dynamic Programming mapping \label{DP-mapping} \label{dynamicmapping}}

In this section we show that for any given belief $b \in \mathcal{B}_2$,
equations (\ref{system},i) and (\ref{system},ii) have unique solutions $u$ and $\alpha$. We also investigate their dependence with respect to $b$.
These equations can be equivalently formulated as follows, with an additional variable $\bar{u} \in (\mathcal{G}_2)^{T+1}$:
\begin{align} 
\bar{u}(t+1, x) = \ & \sup_{\xi \in \mathcal{M}_t} \int_{\mathbb{R}^d} u(t+1, x + y) \dd \xi(y), \label{eq_sys_0} \\
 u(t,x) = & \ \inf_{a \in \mathbb{R}^d}   \frac{1}{2} \vert a \vert^2 + \langle a, P(t,b) \rangle +  F(t,x,b)  + \bar{u}(t+1, x + a), \label{eq_sys_i} \\
 \alpha_t(x) = & \ \underset{a \in \mathbb{R}^d}{\text{argmin}} \frac{1}{2} \vert a \vert^2 + \langle a, P(t,b) \rangle +  F(t,x,b)  + \bar{u}(t+1, x + a), \label{eq_sys_ii} \\
 u(T,x)= \ & F(T,x,b), \label{eq_sys_term}
\end{align}
for all $t \in \T$ and for all $x \in \mathbb{R}^d$.
The first step of our analysis consists in rewriting these equations in a functional form, with the help of the Moreau envelope and the proximal operator (introduced in \eqref{prox-operator}).

\begin{lmm} \label{u-alpha-prox-lemma}
Let $b \in \mathcal{B}_2$.
Let $u \in (\mathcal{G}_2)^{T+1}$, let $\bar{u} \in (\mathcal{G}_2)^{T+1}$, and let $\alpha \in (\mathcal{G}_1 \cap 1\mathrm{-Lip})^T$. Then, for any $t \in \T$ and for any $x \in \mathbb{R}^d$, equations \eqref{eq_sys_0}-\eqref{eq_sys_ii} hold true if and only if
\begin{align}
\bar{u}(t+1,x)= & \ \Upsilon[u(t+1,\cdot)](x), \label{bar_u_ups} \\
u(t,x) = & \ V_{\bar{u}(t+1,\cdot)} (x-P(t,b)) + F(t,x,b) -  \frac{1}{2}|P(t,b)|^2, \label{u-V} \\
\alpha_t(x) = & \ (\prox_{\bar{u}(t+1,\cdot)} - \Id)(x-P(t,b)) - P(t,b). \label{alpha-prox}
\end{align}
\end{lmm}

\begin{proof}
Equality \eqref{bar_u_ups} is obviously equivalent to \eqref{eq_sys_0}, by the definition of $\Upsilon$.
By the change of variable
$y= x+ a$,
the dynamic programming equation \eqref{eq_sys_i} can be reformulated as follows:
\begin{align} 
u(t,x) & = \inf_{y \in \mathbb{R}^d} \Big( \frac{1}{2} \vert y-x \vert^2 + \langle (y - x), P(t,b) \rangle + F(t,x,b) + \bar{u}(t+1,y) \Big) \notag \\
& = \inf_{y \in \mathbb{R}^d} \left(  \frac{1}{2} \vert y-(x-P(t,b)) \vert^2 + \bar{u}(t+1,y) \right) + F(t,x,b) - \frac{1}{2}|P(t,b)|^2. \label{eq:u_star}
\end{align}
This proves the equivalence between \eqref{eq_sys_i} and \eqref{u-V}.
Moreover, since $\bar{u}(t+1,\cdot)$ is convex, the right-hand side of \eqref{eq:u_star} has a unique minimizer given by
\begin{equation*}
y^*:= \prox_{\bar{u}(t+1,\cdot)}(x-P(t,b))
\end{equation*}
and therefore, the unique minimizer in the right-hand side of \eqref{eq_sys_ii} is $y^*-x$, which proves the equivalence between \eqref{eq_sys_ii} and \eqref{alpha-prox}.
The lemma is proved.
\end{proof}

\begin{lmm} \label{alpha-u-in-G1-G2}
Let $b \in \mathcal{B}_2$. There exists a unique triplet $(u,\bar{u},\alpha) \in (\mathcal{G}_2)^{T+1} \times (\mathcal{G}_2)^{T+1} \times (\mathcal{G}_1 \cap 1\mathrm{-Lip})^T$ such that \eqref{eq_sys_0}-\eqref{eq_sys_term} holds true.
Moreover, for any $t \in \TT$, we have
\begin{equation} \label{eq:estim_glob_u}
u(t,\cdot) \in \mathcal{Q}^{C_u}_2,
\end{equation}
and for any $t \in \T$, we have
\begin{equation} \label{eq:estim_glob_alpha}
\alpha_t(\cdot) \in \mathcal{G}_1^{C_\alpha} \cap 1\mathrm{-Lip},
\end{equation}
for some positive constants $C_\alpha$ and $C_u$ independent of $t$ and $b$. 
\end{lmm}

\begin{proof}
Since $u(T,\cdot)$ is uniquely defined by the terminal condition \eqref{eq_sys_term}, $\bar{u}(T,\cdot)$ is uniquely defined by \eqref{bar_u_ups} (with $t=T-1$). Then $u(T-1,\cdot)$ and $\alpha_{T-1}(\cdot)$ are uniquely defined by \eqref{u-V} and \eqref{alpha-prox} (with $t=T-1$) and so on, until $t=0$.

Let us prove \eqref{eq:estim_glob_u} by backward induction. The terminal condition $u(T,\cdot)= F(T,\cdot,b)$ and Assumption \ref{B} ({\normalfont \text{i}}) imply that $u(T,\cdot) \in \mathcal{Q}_2^C$, for some constant $C> 0$ (independent of $b$).
Let us take $t \in \mathcal{T}$ and let us suppose that $u(t+1,\cdot) \in \mathcal{Q}_2^C$. Then by Lemma \ref{Ups} and relation \ref{assum-Mt}, we have $\bar{u}(t,\cdot) \in \mathcal{Q}^{C}_2$.
Recall that by Lemma \ref{u-alpha-prox-lemma}, we have
\begin{equation} \label{eq:inter_glob_u}
u(t,\cdot) = V_{\bar{u}(t+1, \cdot)} (\cdot-P(t,b)) + F(t,\cdot,b) -  \frac{1}{2}|P(t,b)|^2.
\end{equation}
By Assumptions \ref{B} (i) and (iv), $F(t,\cdot,b) - \frac{1}{2} |P(t,b)|^2 \in \mathcal{Q}_2^C$. Using again Assumption \ref{B} (iv) and Lemma \ref{V-in-Q2}, we obtain that $V_{\bar{u}(t+1,\cdot)} (\cdot-P(t,b)) \in \mathcal{Q}_2^C$. Therefore, the right-hand side of \eqref{eq:inter_glob_u} lies in $\mathcal{Q}_2^C$ and finally, $u(t,\cdot) \in \mathcal{Q}_2^C$, where $C$ is independent of $b$.

Let us prove \eqref{eq:estim_glob_alpha}.
By Lemma \ref{u-alpha-prox-lemma}, we have
\begin{equation} \label{eq:repeat_alpha_prox}
\alpha_t(\cdot) = (\prox_{\bar{u}(t+1,\cdot)} - \Id)( \cdot -P(t,b)) - P(t,b).
\end{equation}
We already know that $\bar{u}(t+1,\cdot) \in \mathcal{Q}_2^C$. Moreover, by Assumption \ref{B} (iv), $P(t,b)$ is bounded. Therefore, by Lemma \ref{prox-in-G1}, $\prox_{\bar{u}(t+1,\cdot)}(\cdot- P(t,b)) \in \mathcal{G}_1^C$. Then it is easy to show that $\alpha_t(\cdot,b) \in \mathcal{G}_1^C$, where again, $C$ does not depend on $b$. Finally, $\alpha(t,\cdot)$ is non-expansive as a consequence of \eqref{eq:repeat_alpha_prox} and Proposition \ref{non-expensive}. The lemma is proved.
\end{proof}

From now on, we denote by $(u^*(\cdot,\cdot,b),\bar{u}^*(\cdot,\cdot,b),\alpha_{\cdot}^*(\cdot,b))$
the unique solution to (\ref{system},i)-(\ref{system},ii).

\begin{lmm} \label{u-d1}
There exists $C>0$  such that for any $(t,b_1,b_2) \in \T \times \mathcal{B}_2 \times \mathcal{B}_2$,
  \begin{align} \label{ub1-ub2-G2}
& \|u^\star(t,\cdot , b_1) -u^\star(t,\cdot , b_2)\|_{\mathcal{G},2}  \leq  C d_1(b_1,b_2), \\
 & \|\bar{u}^\star(t,\cdot , b_1) -\bar{u}^\star(t,\cdot , b_2)\|_{\mathcal{G},2}  \leq  C d_1(b_1,b_2). \label{bar-ub1-ub2-G2}
 \end{align}
\end{lmm}

\begin{proof}
In the proof, all constants $C$ are independent of $b_1$ and $b_2$.
We proceed by backward induction. By Assumption \ref{B} ({\normalfont \text{iii}}) and by the terminal condition $u^\star(T,\cdot,b)= F(T,\cdot,b)$, inequality \eqref{ub1-ub2-G2} holds true for $t= T$.
Let $t \in \T$. Suppose that
\begin{equation*}
\|u^\star(t+1,\cdot , b_1) - u^\star(t+1,\cdot , b_2)\|_{\mathcal{G},2} \leq C d_1(b_1,b_2),
\end{equation*}
for some positive constant $C>0$ independent of $b_1$ and $b_2$. By Remark \ref{M} and Lemma \ref{Ups}, we deduce that
\begin{equation} \label{assumption-ubar_rec}
\| \bar{u}^\star(t+1,\cdot , b_1) - \bar{u}^\star(t+1,\cdot , b_2)\|_{\mathcal{G},2} \leq C d_1(b_1,b_2).
\end{equation}
By Lemma \ref{u-alpha-prox-lemma}, we have
\begin{equation} \label{ustar-ustar_rec}
u^\star(t,x , b_1) -u^\star(t,x , b_2)  =    a_1(t,x,b_1,b_2) +  a_2(t,x,b_1,b_2) +  a_3(t,x,b_1,b_2),
\end{equation}
where
\begin{align*}
a_1(t,x,b_1,b_2) &:= V_{\bar{u}^\star(t+1,\cdot , b_1)}(x-P(t,b_1)) - V_{\bar{u}^\star(t+1,\cdot, b_2)}(x-P(t,b_1)),\\
a_2(t,x,b_1,b_2) &:= V_{\bar{u}^\star(t+1,\cdot, b_2)}(x-P(t,b_1)) - V_{\bar{u}^\star(t+1,\cdot, b_2)}(x-P(t,b_2)),\\
a_3(t,x,b_1,b_2) &:=F(t,x,b_1) - F(t,x,b_2) + \frac{1}{2}(|P(t,b_2)|^2  -  |P(t,b_1)|^2).
\end{align*}
It remains to bound $a_1(t,\cdot,b_1,b_2)$, $a_2(t,\cdot,b_1,b_2)$, and $a_3(t,\cdot,b_1,b_2)$ in $\mathcal{G}_2^C$.
We deduce from Lemma \ref{Vu-Vv}, Assumption \ref{B} (iv), and estimate \eqref{assumption-ubar_rec}, that
\begin{align*}
|a_1(t,x,b_1,b_2)| &\leq \|V_{\bar{u}^\star(t+1,\cdot , b_1)} - V_{\bar{u}^\star(t+1,\cdot, b_2)}\|_{\mathcal{G},2}(1+|x-P(t,b_1)|^2) \\
& \leq C\|\bar{u}^\star(t+1,\cdot , b_1) - \bar{u}^\star(t+1,\cdot, b_2)\|_{\mathcal{G},2}(1+|x|^2)\\
& \leq  C d_1(b_1,b_2)(1+|x|^2).
\end{align*}
Then by Lemma \ref{Vu-xy} and Assumption \ref{B} ({\normalfont \text{iv}}), we have
\begin{align*}
|a_2(t,x,b_1,b_2)| &\leq C(1+|x-P(t,b_1)| + |x-P(t,b_2)|)|P(t,b_2)-P(t,b_1)|\\
& \leq C(1+|x|)d_1(b_1,b_2) \\
& \leq C(1+ |x|^2) d_1(b_1,b_2).
\end{align*}
Finally by Assumption \ref{B} ({\normalfont \text{ii-iv}}), we have
\begin{align*}
 |a_3(t,x,b_1,b_2)|& \leq \|F(t,\cdot,b_1)-F(t,\cdot,b_2)\|_{\mathcal{G},2}(1+|x|^2) + C | P(t,b_1)-P(t,b_2) |\\
 & \leq C(1+|x|^2)d_1(b_1,b_2).
\end{align*}
Then combining \eqref{ustar-ustar_rec} and the three estimates of $a_1$, $a_2$, and $a_3$, we obtain that
\begin{equation} \nonumber
\|u^\star(t,\cdot , b_1) -u^\star(t,\cdot , b_2)\|_{\mathcal{G},2}  \leq  C d_1(b_1,b_2),
\end{equation}
which concludes the proof.
\end{proof}

\begin{lmm} \label{alpha-d1}
There exists $C>0$ such that for any $(t,b_1,b_2) \in \T \times \mathcal{B}_2 \times \mathcal{B}_2$,
\begin{align} \label{eq:alpha_d1}
\| \alpha^\star_t(\cdot , b_1) -  \alpha^\star_t(\cdot , b_2) \|_{\mathcal{G},1} \leq C \left(d_1(b_1,b_2)^{1/2} + d_1(b_1,b_2)\right).
\end{align}
\end{lmm}

\begin{proof}
Let $(t,b_1,b_2) \in \T \times \mathcal{B}_2 \times \mathcal{B}_2$.
By Lemma \ref{u-alpha-prox-lemma}, we have
\begin{equation} \label{eq:var_alpha}
\alpha^\star_t(x,b_1) - \alpha^\star_t(x,b_2)
= a_4(t,x,b_1,b_2) + a_5(t,x,b_1,b_2),
\end{equation}
where
\begin{align*}
a_4(t,x,b_1,b_2)= \ & \prox_{\bar{u}^\star(t+1,\cdot, b_1)} (x-P(t,b_1)) -  \prox_{\bar{u}^\star(t+1,\cdot , b_2)} (x-P(t,b_1)), \\
a_5(t,x,b_1,b_2)= \ & \prox_{\bar{u}^\star(t+1,\cdot , b_2)} (x-P(t,b_1)) -  \prox_{\bar{u}^\star(t+1,\cdot , b_2)} (x-P(t,b_2)).
\end{align*}
Using successively Lemma \ref{QinH}, Assumption \ref{B} (iv), and estimate \eqref{bar-ub1-ub2-G2}, we obtain
\begin{align*}
|a_4(t,x,b_1,b_2)|
\leq \ & \| \prox_{\bar{u}^\star(t+1,\cdot, b_1)} - \prox_{\bar{u}^\star(t+1,\cdot , b_2)} \|_{\mathcal{G},1} (1 + |x-P(t,b_1|) \\
\leq \ & C \| \bar{u}^\star(t+1,\cdot, b_1) - \bar{u}^\star(t+1,\cdot , b_2) \|^{1/2} (1 + |x|) \\
\leq \ & C \| d_1(b_1,b_2) \|^{1/2} (1 + |x|).
\end{align*}
Moreover, since $\prox_{\bar{u}^\star(t+1,\cdot , b_2)}$ is non-expansive, we have with Assumption \ref{B} (iii) that
\begin{align*}
|a_5(t,x,b_1,b_2)|
\leq |(x-P(t,b_1)) -(x-P(t,b_2))| \leq d_1(b_1,b_2).
\end{align*}
Combining the two obtained estimates of $a_4$ and $a_5$ with \eqref{eq:var_alpha}, we obtain \eqref{eq:alpha_d1}.
\end{proof}

\subsection{Kolmogorov mapping \label{Kolmo-mapping}}

We study now the Kolmogorov mapping
\begin{equation*}
(\mathcal{G}_1^{C_\alpha} \cap 1\mathrm{-Lip})^{T} \ni \alpha \mapsto (m^\star,\mu^\star,b^\star)(\alpha),
\end{equation*}
where $(m^\star,\mu^\star,b^\star)$ is the solution to (\ref{system},iii-v).
\begin{lmm} \label{kolmogorov-reg}
There exists $C_b>0$ such that for any $\alpha \in (\mathcal{G}_1^{C_\alpha} \cap 1\mathrm{-Lip})^T $,
\begin{align*}
m^\star(\alpha) \in \big(\mathcal{P}^{C_{b}}_2(\mathbb{R}^d) \big)^{T+1}, \quad
\mu^\star(\alpha) \in  \big( \mathcal{P}^{C_{b}}_2(\mathbb{R}^{2d}) \big)^{T}, \quad \text{and} \quad
b^\star(\alpha) \in \mathcal{B}^{C_{b}}_2.
\end{align*}
In addition the three mappings $m^\star,\mu^{\star}$ and $b^\star$ are continuous.
\end{lmm}

\begin{proof}
Let $\alpha \in (\mathcal{G}_1^{C_\alpha} \cap 1\mathrm{-Lip})^T $. All constants $C$ in the proof are independent of $\alpha$. Let us first prove by induction that for any $t \in \TT$, there exists a constant $C>0$ independent of $\alpha$ such that $m^\star(t,\cdot,\alpha) \in \mathcal{P}_2^C(\mathbb{R}^d)$ and such that, $m^\star(t,\cdot,\alpha)$ is continuous with respect to $\alpha$.
The claim is clear for $t=0$, since $m^\star(0,\cdot,\alpha)= \bar{m} \in \mathcal{P}_2^C(\mathbb{R}^d)$, by Assumption \ref{A}. Now, let us assume that the claim holds true for some $t \in \T$.
We recall that
\begin{equation*}
m^\star(t+1,\cdot,\alpha)
= \nu(t) \ast \left[
(id+\alpha_t) \sharp m^\star(t,\cdot,\alpha)
\right].
\end{equation*}
Since $\nu(t) \in \mathcal{P}_2^C(\mathbb{R}^d)$ (by Assumption \ref{A}) and since $\alpha_t \in \mathcal{G}_1^{C_\alpha} \cap 1\mathrm{-Lip}$, we obtain with Lemma \ref{convol-p1} and Lemma \ref{g-sharp-m} that $m^\star(t+1,\cdot,\alpha) \in \mathcal{P}_2^C(\mathbb{R}^d)$ and that $m^\star(t+1,\cdot,\alpha)$ is a continuous function of $\alpha$, by composition.

It remains to justify the boundedness of $\mu^*$ and $b^*$. We recall that for any $t \in \T$,
\begin{equation*}
\mu^\star(t,\cdot,\alpha)
= (id,\alpha_t) \sharp m^\star(t,\cdot,\alpha).
\end{equation*}
We deduce from Lemma \ref{g-sharp-m} that $\mu^\star(t,\cdot,\alpha) \in \mathcal{P}_2^C(\mathbb{R}^{2d})$ and that $\mu^\star(t,\cdot,\alpha)$ is a continuous function of $\alpha$, by composition.
It immediately follows that $b^\star(\alpha) \in \mathcal{B}_2^C$ and that $b$ is continuous.
\end{proof}

\subsection{Existence of equilibrium \label{existence-eq}}

We are ready to prove the existence of a solution of system \eqref{system}. The proof relies on the Schauder fixed point theorem, that we first recall.

\begin{thrm} \label{Schauder} (Schauder)
Let $C$ be a convex and compact set in a Banach space $X$, and let $T \colon C \to C$ be a continuous mapping. Then $T$ has a fixed point, i.e. there exists $x \in C$ such that 
\begin{equation*}
T(x) = x.
\end{equation*}
\end{thrm}

\begin{thrm} \label{ex}
There exists $(u,\alpha,m,\mu,b) \in \big( \mathcal{G}_2^{C_u} \big)^T \times \big( \mathcal{G}_1^{C_\alpha} \cap 1\mathrm{-Lip} \big)^T \times \big( \mathcal{P}^{C_{b}}_2(\mathbb{R}^d) \big)^{T+1} \times \big( \mathcal{P}^{C_{b}}_2(\mathbb{R}^{2d}) \big)^{T} \times \mathcal{B}^{C_{b}}_2$ solution to system {\normalfont(\ref{system})}, where $C_u$, $C_\alpha$ and $C_b$ are the constants obtained in Lemma \ref{alpha-u-in-G1-G2} and Lemma \ref{kolmogorov-reg}.
\end{thrm}

\begin{proof}
By Lemma \ref{alpha-d1} and Lemma \ref{kolmogorov-reg}, the mapping
\begin{equation*}
\mathcal{B}_2^{C_b} \ni b \mapsto b^\star \circ \alpha^\star(b) \in \mathcal{B}_2^{C_b}
\end{equation*}
is continuous for the distance $d_1$. Moreover, $\mathcal{B}_2^{C_b}$ is compact for $d_1$, see \cite[Lemma 25]{laurent-variationnal}. Therefore, by the Schauder fixed point theorem, there exists $\bar{b} \in \mathcal{B}_2^{C_b}$ such that $\bar{b}= b^\star \circ \alpha^\star(\bar{b})$.
Let us set $\bar{u}= u^\star(\bar{b})$, $\bar{\alpha}= \alpha^\star(\bar{b})$, $\bar{m}= m^\star(\bar{\alpha})$, and $\bar{\mu}= \mu^\star(\bar{\alpha})$. Then $(\bar{u},\bar{\alpha},\bar{m},\bar{\mu},\bar{b})$ is solution to (\ref{system}) and lies in the announced set.
\end{proof}

\section{Connection with a finite player game \label{section-6}} 

In this section we establish a connection between the coupled system \eqref{system} and a dynamic game with $N$ players. 
More precisely, we fix a solution $(\bar{u},\bar{\alpha},\bar{m},\bar{\mu},\bar{b})$ of system (\ref{system}) and consider the situation where each of the $N$ players adopts the feedback $\bar{\alpha}$. We show that this situation is an $\varepsilon$-Nash equilibrium for the $N$-player game and we quantify the rate of convergence of $\varepsilon$ to 0 as $N$ goes to infinity.

To show this, the following restriction on Assumption \ref{B} (ii) will be required, in particular to prove Lemma \ref{abs-rho-F}.

\begin{assumption}
\label{C} There exists $C>0$ such that for any $t\in \T$ and for any $b_1$ and $b_2$ in $\mathcal{B}_2$,
\begin{equation} \begin{array}{cl}
({\normalfont \text{i}}) & F(t,\cdot, b_1) \in \mathcal{Q}^C_1,\\
({\normalfont \text{ii}}) & \|F(t,\cdot, b_1) - F(t,\cdot, b_2)\|_{\mathcal{G},1} \leq C d_1(b_1,b_2).
\end{array} \nonumber
\end{equation}
\end{assumption}

We have already fixed a solution to system \eqref{system}, now we also fix the number of players $N$; all constants $C$ appearing in the sequel are independent of $N$.

\subsection{Formulation of the game}

Let $\mathcal{N} := \{1, \ldots, N\}$ and let $i\in \mathcal{N}$. For any vector $(x^1,\ldots,x^N)$ we denote
\begin{align*}
& \bm{x}  = (x^1,\ldots,x^N), \\
& \bm{x}^{-i} =  (x^1,\ldots,x^{i-1},x^{i+1},\ldots,x^N).
\end{align*}
Consider a probability space $(\Omega, \bm{\mathcal{F}}, \mathbb{P})$. Let $(X^i_0)_{i\in \mathcal{N}}$ be i.i.d.\@ random variables with law $\mathcal{L}(X^i_0) = \bar{m}$. Let $(Y^i_{t})_{i \in \mathcal{N}, t\in \T}$ be independent random variables, independent of $(X^i_0)_{i\in \mathcal{N}}$, with law $\mathcal{L}(Y^i_t) = \nu(t)$. We denote $\bm{\nu}(t) := \bigotimes_{i = 1}^N \nu(t)$.
We define the filtration $(\bm{\mathcal{F}}_t)_{t\in \TT}$ as follows: $\bm{\mathcal{F}}_0 := \sigma(\bm{X}_0)$ is the sigma-algebra generated by $\bm{X}_0$, $\bm{\mathcal{F}}_{t+1} := \sigma(\bm{X}_0, \bm{Y}_{[t]})$. In this section we denote 
\begin{equation*}
 \bm{L}_t^p(\Omega,\mathbb{R}^{d'}) := L^p(\Omega,\bm{\mathcal{F}}_t, \mathbb{P},\mathbb{R}^{d'}),
\end{equation*}
the space of $\bm{\mathcal{F}}_t$ measurable random variables with finite $p$-th order moment and value in $\mathbb{R}^{d'}$. When the dimension is $d'=1$, we simplify the notation $ \bm{L}_t^p =  \bm{L}_t^p(\Omega,\mathbb{R})$.
For any $t\in \T$, we consider the control set
\begin{equation*}
\bm{\mathcal{A}}_t := \bm{L}_t^2(\Omega,\mathbb{R}^d), \qquad \bm{\mathcal{A}} := \bm{\mathcal{A}}_0 \times \cdots \times \bm{\mathcal{A}}_{T-1}, \qquad \bm{\mathcal{A}}^N := \prod_{i=1}^N\bm{\mathcal{A}}.
\end{equation*}
For any $t\in \T$ and for any constant $C>0$ we denote  $\bm{\mathcal{A}}^C_t$ the set of controls $A \in \bm{\mathcal{A}}_t$ such that
\begin{equation*}
\int_{\Omega}|A(\omega)|^2  \dd \mathbb{P}(\omega)\leq C
\end{equation*}
and we set $\bm{\mathcal{A}}^C := \bm{\mathcal{A}}^C_0 \times \cdots \times \bm{\mathcal{A}}^C_{T-1}$. The control of player $i\in \mathcal{N}$ is an adapted stochastic process $A^i \in \bm{\mathcal{A}}$, whose associated trajectory $(X_t^i[A^i])_{t\in \TT}$ is defined by the following state equation
\begin{equation*}
X^i_{t+1} = X^i_{t} + A^i_t + Y^i_{t}.
\end{equation*}

\begin{rmrk} \label{X0-rem}
Let $R>0$. There exists $C>0$ (depending on $R$) such that for any $i \in \mathcal{N}$ and for any $A^i \in \bm{\mathcal{A}}^R$, $\mathbb{E}\left[|X^i_t[A^i]|^2 \right] \leq C$ for any $t \in \TT$, since $\mathcal{L}(X_0^i) \in \mathcal{P}_2(\mathbb{R}^d)$ and  $\mathcal{L}(Y_t^i)\in \mathcal{P}_2(\mathbb{R}^d)$.
\end{rmrk}

Given $\bm{A} \in \bm{\mathcal{A}}^N$, we define the random empirical measure of the positions and the random empirical joint measure of the positions and actions of players by
  \begin{equation*}
 m_{\bm{A}}^{N}(t) := \frac{1}{N} \sum_{i \in \mathcal{N}}\delta_{X^i_t[A^i]}, \qquad  \mu_{\bm{A}}^{N}(t) := \frac{1}{N}\sum_{i \in \mathcal{N}} \delta_{(X^i_t[A^i],A^i_t)},
 \end{equation*}
where $\delta$ denotes the Dirac measure. We set
\begin{equation*}
b_{\bm{A}}^{N} := \left( \mu_{\bm{A}}^{N}(0),\ldots, \mu_{\bm{A}}^{N}(T-1), m_{\bm{A}}^{N}(T) \right).
\end{equation*}
For any $i \in \mathcal{N}$ and for any $t\in \T$, we define the \textit{individual conditional risk measure} $\rho^i _{t} \colon \bm{L}^1_{t+1} \to  \bm{L}^1_{t}$,
\begin{equation*} \label{riskmeasure-N-players-11}
\rho^i_{t}(U_{t+1})(\bm{x}_0,\bm{y}_{[t-1]}) = \sup_{Z \in \mathcal{Z}_t} \int_{\Omega} U_{t+1}(\bm{x}_0,\bm{y}_{[t-1]},\bm{Y}_t(\omega)) Z(Y^i_t(\omega)) \dd \mathbb{P}(\omega).
\end{equation*}
We define the set
\begin{equation*}
\bm{Q}^i_{t+1} := \left\{Q \in \bm{L}^{\infty}_{t+1}, \, Q = Z(Y^i_t) \text{ a.s.}, \, Z \in \mathcal{Z}_t\right\}.
\end{equation*}
Then $\rho^i_t$ can be expressed in the following form:
\begin{equation} \label{rho-def-Q}
\rho^i_{t}(U_{t+1})  = \sup_{Q_{t+1} \in \bm{Q}^i_{t+1}} \mathbb{E} \left[U_{t+1} Q_{t+1} \vert \bm{\mathcal{F}}_t \right].
\end{equation}
In addition we have that
\begin{align*}
\rho^i_{t}(U_{t+1})(\bm{x}_0,\bm{y}_{[t-1]}) & =  \sup_{\xi \in \mathcal{M}_{t}} \int_{\mathbb{R}^{d}}\int_{\mathbb{R}^{Nd}} U_{t+1}(\bm{x}_0,\bm{y}_{[t]}) \dd  \bm{\nu}^{-i}(t,\bm{y}_t^{-i}) \dd \xi(y^i_t),
\end{align*}
where $\bm{\nu}^{-i}(t) := \bigotimes_{j \in \mathcal{N}\setminus\{i\}}^N \nu (t)$.
Then $(\rho^i_t)_{t \in \T}$ is a family of conditional risk mappings. We define the associated individual composite risk measure $\rho^i \colon \bm{L}^1_{T}  \to \mathbb{R}$,

\begin{equation*}
\rho^i(U) := \mathbb{E} \left[ \rho^i_{0}  \circ \cdots  \circ  \rho^i_{ T-1}(U) \right].
\end{equation*}
Here players are risk averse with respect to their individual noise only.
For any $\bm{A} \in\bm{\mathcal{A}}^N$ the cost of the player $i \in \mathcal{N}$ is given by
\begin{align*}
\mathcal{J}^{i,N}(A^i,\bm{A}^{-i}) := \rho^i \left( \sum_{t=0}^{T-1} \ell(t,{X}^i_t[A^i],A^i_t,b^N_{\bm{A}}) + F(T,{X}^i_T[A^i],b^N_{\bm{A}}) \right).
\end{align*}

\begin{dfntn}
Let $\varepsilon \geq 0$. We say that an $N$-uplet $\widehat{{\bm{A}}} \in\bm{\mathcal{A}}^N$ is an $\varepsilon$-Nash equilibrium for the $N$-player game if for any $i\in \mathcal{N}$,
\begin{equation} \label{eq:nash_eps_pb}
\mathcal{J}^{i,N}(\widehat{A}^i, \widehat{\bm{A}}^{-i}) \leq \inf_{A^i \in \bm{\mathcal{A}}} \mathcal{J}^{i,N}(A^i, \widehat{\bm{A}}^{-i}) + \varepsilon.
\end{equation}
\end{dfntn}

For $\varepsilon= 0$, we recover the usual definition of a Nash equilibrium.

\subsection{An approximate Nash equilibrium}

For any player $i\in \mathcal{N}$,
we denote by $(\bar{X}_t^i)_{t\in \TT}$ the solution to the closed-loop system
\begin{equation*}
X^i_{t+1} = X^i_{t} + \bar{\alpha}_t(X^i_{t}) + Y^i_{t}.
\end{equation*}
We define the control $\bar{A}^i \in \mathcal{A}$ by
\begin{equation} \label{open-loop-control}
\bar{A}^i_t = \bar{\alpha}_t(\bar{X}^i_t).
\end{equation}
Since $\bar{X}^i_t$ is adapted to $\bm{\mathcal{F}}_t$, the control $\bar{A}_t^i$ is also $\bm{\mathcal{F}}_t$-measurable. Moreover, $\bar{\alpha}_t$ is $1$-Lipschitz and the random variables $X_0$ and $(Y_t)_{t\in \T}$ have a bounded second-order moment, thus $\bar{A}^i \in \bm{\mathcal{A}}$. In addition, by Proposition \ref{proposition:dyn_prog}, $\bar{A}^i$ minimizes the following cost $\mathcal{J}^i$:
\begin{equation} \label{eq:valueJi}
\mathcal{J}^i(A^i,\bar{b}) := \rho^i \left( \sum_{t = 0}^{T-1}  \ell (t, X^i_t[A^i] , A^i_t, \bar{b})  +F(T,X^i_T[A^i], \bar{b}) \right).
\end{equation}
Finally we set $\bar{\bm{A}} = (\bar{A}^1,\ldots,\bar{A}^N)$.
The following result states that $\bar{\bm{A}}$ is an $\varepsilon$-Nash equilibrium.

\begin{thrm} \label{thm:approx_nash}
Let $\xi \in (0,1/2)$. There exists a constant $C>0$, independent of $N$, such that the $N$-uplet $\bar{\bm{A}}$ defined above is an $\varepsilon$-Nash equilibrium with
\begin{equation*}
\varepsilon= C N^{-\tau(d)/2}, \qquad
\tau(d) =  \begin{cases} 1/2 - \xi & \text{ if } d \in \{1,2\},\\
1/d & \text{ if } d \geq 3.
\end{cases}
\end{equation*}
In addition we have that
\begin{equation} \label{eq:convergence_cost}
| \mathcal{J}^{i,N}(\bm{\bar{A}})  -  \mathcal{J}^{i}(\bar{A}^i,\bar{b})| \leq  CN^{-\tau(d)/2},
\end{equation}
for any $i \in \mathcal{N}$.
\end{thrm}

The proof of the theorem can be found at the end of Subsection \ref{subsection:proof_thm} (page \pageref{proof:main_thm}), which contains technical intermediate lemmas.
They rely on the following result.

\begin{thrm} (Fournier-Guillin) \label{fournier-thm}
Let $c>0$, let $\xi \in (0,1/2)$, and let $\mu \in \mathcal{P}^c_{2}(\mathbb{R}^d)$. Consider $N$ i.i.d.\@ random variables $(X_i)_{i\in\{1,\ldots,N\}}$ in $\mathbb{R}^d$ with law $\mu$ and denote by $\mu_N$ their empirical measure, defined by
\begin{equation}
\mu_N = \frac{1}{N} \sum_{i = 1}^N \delta_{X_i}.
\end{equation}
There exists a constant $C>0$ depending only on $c$, $d$, and $\xi$ such that
\begin{equation*}
\mathbb{E}\left[d_1(\mu,\mu_N))\right] \leq CN^{-\tau(d)}.
\end{equation*}
\end{thrm}

\begin{proof}
The theorem is a direct application of \cite[Theorem 1]{Fournier2015} with $q = \frac{2}{1+2\xi}$ if $d \in \{1,2\}$ and $q = 2$ if $d \geq 3$.
\end{proof}

\subsection{Proof of Theorem \ref{thm:approx_nash}} \label{subsection:proof_thm}

We begin with four technical lemmas dealing with the regularity of the individual risk measures $\rho^i$.

\begin{lmm} \label{subadditive}
For any player $i\in \mathcal{N}$ the risk measure $\rho^i$ is subadditive, that is
\begin{equation*}
\rho^i(U+V) \leq \rho^i(U) + \rho^i(V),
\end{equation*}
for any $U$ and $V$ in $\bm{L}_T^1$.
\end{lmm}

\begin{proof}
Let us define $\pi_T^i(U)= U$ and $\pi_t^i(U)= \rho_t^i \circ \rho_{t+1}^i ... \circ \rho_{T-1}^i(U)$, for any $U \in \bm{L}_T^1$.
Note that $\pi_t^i = \rho_t^i \circ \pi_{t+1}^i$, for any $t \in \T$.
We prove by backward induction that $\pi_t^i$ is subadditive for any $t \in \TT$. The claim is trivial for $t= T$. Let $t \in \T$. Assume that $\pi_{t+1}^i$ is subadditive, let us prove that $\pi_t^i$ is subadditive.
First we observe that for any $U$ and $V$ in $\bm{L}_T^1$,
\begin{align*}
\rho^i_t(U + V) & = \sup_{Q \in \bm{Q}^i_{t+1}} \mathbb{E} \left[(U + V) Q \vert \bm{\mathcal{F}}_t \right] \\ & \leq \sup_{Q \in \bm{Q}^i_{t+1}} \mathbb{E} \left[U Q \vert \bm{\mathcal{F}}_t \right] + \sup_{Q \in \bm{Q}^i_{t+1}} \mathbb{E} \left[ V Q \vert \bm{\mathcal{F}}_t \right]
 = \rho^i_t(U) + \rho^i_t(V), \quad \text{a.s.}
\end{align*}
It follows with the monotonicity of $\rho_t^i$ that
\begin{align*}
\pi^i_t(U+V)
= \ & \rho_t^i \circ \pi^i_{t+1}(U+V) \\
\leq \ & \rho_t^i ( \pi_{t+1}^i(U) + \pi_{t+1}^i(V)) \\
\leq \ & \rho_t^i \circ \pi_{t+1}^i(U)
+ \rho_t^i \circ \pi_{t+1}^i(V)
= \pi^i_t(U) + \pi_t^i(V), \quad \text{a.s.}
\end{align*}
Recalling that $\rho^i(U) = \mathbb{E}\left[\rho^i_0 \circ \cdots \circ \rho^i_{T-1}(U) \right]= \mathbb{E} \left[ \pi_0(U) \right]$, we conclude that $\rho^i$ is also subadditive.
\end{proof}

The following result is close to a triangle inequality for risk measures. The difference with the triangle inequality is due to the positive homogeneity of risk measures, while norms are absolutely homogeneous.

\begin{lmm} \label{ineq-rho-triangle}
For any $i\in \mathcal{N}$ and for any $U$ and $V$ in $\bm{L}_T^1$, we have
\begin{equation} \label{eq:ineq_tri}
\left|\rho^i(U+V) - \rho^i(U)\right| \leq \rho^i(|V|).
\end{equation}
\end{lmm}

\begin{proof}
By the subadditivity and by the monotonicity of $\rho^i$, we have
\begin{equation*}
\rho^i(U+V) - \rho^i(U)
\leq (\rho^i(U) + \rho^i(V)) - \rho^i(U)
= \rho^i(V)
\leq \rho^i(|V|).
\end{equation*}
Similarly, we have
\begin{equation*}
\rho^i(U) - \rho^i(U+V)
= \rho^i(U+V-V) - \rho^i(U+V)
\leq \rho^i(-V)
\leq \rho^i(|V|).
\end{equation*}
Inequality \eqref{eq:ineq_tri} follows.
\end{proof}

\begin{lmm} \label{rho-esp-lemma}
There exists $C>0$ such that for any $(i,t) \in \mathcal{N} \times \T$ and for any $U \in \bm{L}_T^1$,
\begin{equation} \label{ineq-esp-rho-lemma}
\frac{1}{C} \mathbb{E} \left[ |U| \right]
\leq \rho^i(U)
\leq C \mathbb{E} \left[ |U| \right].
\end{equation}
\end{lmm}

\begin{proof}
All constants $C$ in the proof are independent of $U$.
Recall the definition of $\pi_t^i$, introduced in the proof of Lemma \ref{subadditive}.
We prove by backward induction that for any $t \in \TT$, there exists $C>0$ such that for any $U \in \bm{L}^1_T$,
\begin{equation*}
\frac{1}{C} \mathbb{E} \left[
|U| \, \big| \bm{\mathcal{F}}_t
\right]
\leq \pi_t^i(U)
\leq C \mathbb{E} \left[
|U| \, \big| \bm{\mathcal{F}}_t
\right], \quad \text{a.s.}
\end{equation*}
The claim is trivial for $t= T$. Let $t \in \T$. Assume that the claim holds true for $t+1$. We first observe that for any $U \in \bm{L}_{t+1}^1$,
\begin{equation}
\frac{1}{C} \mathbb{E} \left[ |U| \big| \bm{\mathcal{F}}_t \right]
\leq \rho_t^i(U) \leq C \mathbb{E} \left[ |U| \big| \bm{\mathcal{F}}_t \right], \quad \text{a.s.},
\end{equation}
as a direct consequence of Assumption \ref{Z}. It follows with the monotonicity of $\rho_t^i$ that
\begin{align*}
\pi_t^i(U)
= \ & \rho_t^i \circ \pi_{t+1}^i(U) \\
\leq \ & \rho_t^i \left(
C \mathbb{E} \left[ |U| \, \big| \bm{\mathcal{F}}_{t+1} \right]
\right) \\
\leq \ & C \mathbb{E} \left[ C \mathbb{E} \left[ |U| \, \big| \bm{\mathcal{F}}_{t+1} \right] \, \big| \bm{\mathcal{F}}_t \right]
\leq C \mathbb{E} \left[ |U| \, \big| \bm{\mathcal{F}}_t \right], \quad \text{a.s.}
\end{align*}
Similarly we prove that $\pi_t^i(|U|) \geq \frac{1}{C} \mathbb{E} \left[ |U| \, \big| \bm{\mathcal{F}}_t \right]$ a.s.
Recalling that $\rho^i(U)= \mathbb{E} \left[ \pi_0(U) \right]$, we finally obtain \eqref{ineq-esp-rho-lemma}.
\end{proof}

The following lemma is an estimate of the second-order moment of suboptimal controls (for problem \eqref{eq:nash_eps_pb}).

\begin{lmm} \label{alpha0-AC}
There exists $C>0$ such that for any $i \in \mathcal{N}$, if $\widehat{A}^i$ satisfies
\begin{equation} \label{infnasheq-lemma}
  \mathcal{J}^{i,N}(\widehat{A}^i, \bar{\bm{A}}^{-i}) \leq \inf_{A^i \in \bm{\mathcal{A}}} \mathcal{J}^{i,N}(A^i, \bar{\bm{A}}^{-i}) + 1,
\end{equation}
then $\widehat{A}^i \in \bm{\mathcal{A}}^{C}$.
\end{lmm}

\begin{proof}
Let $i \in \mathcal{N}$ and let  $\widehat{A}^i$ satisfy (\ref{infnasheq-lemma}).
All constants $C$ in the proof are independent of $\widehat{A}^i$.
We have
\begin{equation*}
\mathcal{J}^{i,N}(\widehat{A}^i, \bar{\bm{A}}^{-i}) \leq  \mathcal{J}^{i,N}(0, \bar{\bm{A}}^{-i}) + 1 = \rho^i\left( \sum_{t = 0}^T F\left(t,X^i_t[0],b_{(0,\bar{\bm{A}}^{-i})}^N\right) \right) + 1. \label{vin-pi}
\end{equation*}
By Assumption \ref{B} (i), Lemma \ref{rho-esp-lemma}, and Remark \ref{X0-rem},
\begin{equation*}
\rho^i\left( \sum_{t = 0}^T F\left(t,X^i_t[0],b_{(0,\bar{\bm{A}}^{-i})}^N\right) \right) \leq C \mathbb{E}\left[ T + \sum_{t = 0}^T|X^i_t[0]|^2 \right] \leq C.
\end{equation*}
Therefore,
\begin{equation} \label{vin-C}
\mathcal{J}^{i,N}(\widehat{A}^i, \bar{\bm{A}}^{-i}) \leq C.
\end{equation}
We need now to bound $\mathcal{J}^{i,N}(\widehat{A}^i, \bar{\bm{A}}^{-i})$ from below. We obtain by using successively Lemmma \ref{rho-esp-lemma}, Assumptions \ref{B} (i) and (iv), and Young's inequality that
\begin{equation} \label{eq:lower_bound}
\mathcal{J}^{i,N}(\widehat{A}^i, \bar{\bm{A}}^{-i})
\geq \frac{1}{C} 
\mathbb{E}\left[ \, \sum_{t = 0}^{T-1} \left( \frac{1}{2}
 \vert \widehat{A}^i_t\vert^2 - C \vert \widehat{A}^i_t \vert \right) \right]
 - C
 \geq \frac{1}{C} 
\mathbb{E}\left[ \, \sum_{t = 0}^{T-1}
 \vert \widehat{A}^i_t\vert^2  \right]
 - C.
\end{equation}
We deduce then from \eqref{vin-C} and \eqref{eq:lower_bound} that
$\mathbb{E}\left[ \, \sum_{t = 0}^{T-1}
 \vert \widehat{A}^i_t\vert^2  \right] \leq C,
$
which concludes the proof.
\end{proof}

In the following we fix a constant $c>0$ such that the result of Lemma \ref{alpha0-AC} holds and such that $\bar{A}^i \in \bm{\mathcal{A}}^c$ for any $i \in \mathcal{N}$.
Let $b$ and $b'$ in $\mathcal{B}_2$, for any $(t,t',x) \in \T \times \TT \times \mathbb{R}^d$ we define
\begin{equation*}
\Delta P(t,b, b') := P\left(t,b\right)-P(t,b'), \qquad \Delta F(t',x,b,b') := F(t',x,b)  - F(t',x,b'). 
\end{equation*}
For any $(x,A) \in  \mathbb{R}^{Td} \times \bm{\mathcal{A}}$ we define
\begin{equation*}
\Delta \ell(x,A,b, b') := \sum_{t=0}^{T-1} \langle A_t,\Delta P(t,b,b') \rangle + \sum_{t=0}^{T} \Delta F(t,x_t,b,b').
\end{equation*}

\begin{rmrk} \label{remark-F-holder}
For any $t\in \T$ and for any $b$ and $b'$ in $\mathcal{B}_2$, we have
\begin{equation*}
\| \Delta F\left(t,\cdot,b,b'\right) \|_{\mathcal{G},1} \leq 2C d_1(b,b')^{1/2}.
\end{equation*}
Indeed if $d_1(b,b') \geq 1$, Assumption \ref{C} {\normalfont (\text{i})} yields
\begin{equation*} 
\| \Delta F\left(t,\cdot,b,b'\right) \|_{\mathcal{G},1} \leq 2 \sup_{b \in \mathcal{B}_2} \| F \left(t,\cdot,b\right) \|_{\mathcal{G},1} \leq 2C.
\end{equation*}
If $d_1(b,b') \leq 1$, by Assumption \ref{C} {\normalfont (\text{ii})} we have
\begin{equation*}
\| \Delta F\left(t,\cdot,b,b'\right) \|_{\mathcal{G},1} \leq C d_1(b,b') \leq C d_1(b,b')^{1/2}.
\end{equation*}
\end{rmrk}

In the following lemma we study the convergence of the empirical belief to the reference belief $\bar{b}\in \mathcal{B}_2$.

\begin{lmm} \label{lemma-dist-d1}
There exists $C>0$ such that for any $i\in \mathcal{N}$ and for any $A^i \in \bm{\mathcal{A}}^c$,
\begin{equation} \label{d1-estim-i}
\mathbb{E} \left[ d_1\left(b_{(A^i,\bar{\bm{A}}^{-i})}^{N} ,\bar{b}\right) \right] \leq CN^{-\tau(d)}.
\end{equation}
\end{lmm}

\begin{proof}
Let $i \in \mathcal{N}$ and let $A^i \in \bm{\mathcal{A}}^c$. For any $t\in \T$, we have by the triangle inequality
\begin{equation} \label{d1_interm_1}
d_1\left(\mu^N_{(A^i,\bar{\bm{A}}^{-i})}(t) ,\bar{\mu}(t)\right)  \leq d_1\left(\mu^N_{(A^i,\bar{\bm{A}}^{-i})}(t), \mu^N_{\bar{\bm{A}}}(t) \right) + d_1\left(\mu^N_{\bar{\bm{A}}}(t) ,\bar{\mu}(t)\right).
\end{equation}
Let us consider the first term of the right-hand side.
By definition of the distance $d_1$,
\begin{equation} \label{d1-interm-XX-alpha}
d_1 \left(\mu^N_{(A^i,\bar{\bm{A}}^{-i})}(t), \mu^N_{\bar{\bm{A}}}(t) \right) \leq \frac{1}{N} \left( |X_t^i[A^i]- \bar{X}_t^i| + |A^i_t - \bar{A}^i_t| \right), \quad \text{a.s.}
\end{equation}
Since the controls $\bar{A}^i$ and $A^i$ belong to $ \bm{\mathcal{A}}^c$, the first-order moment of $X_t^i[A^i]$ and $\bar{X}_t$ are finite as a consequence of Remark \ref{X0-rem}, thus
\begin{align} \label{E-X-alpha}
\mathbb{E}\left[ |X_t^i[A^i]- \bar{X}_t^i| + |A^i_t - \bar{A}^i_t| \right] \leq C.
\end{align}
Therefore, by \eqref{d1-interm-XX-alpha} and \eqref{E-X-alpha}, we have
\begin{equation} \label{a_1}
d_1\left(\mu^N_{(A^i,\bar{\bm{A}}^{-i})}(t), \mu^N_{\bar{\bm{A}}}(t) \right) \leq \frac{C}{N}.
\end{equation}
Let us consider now the second-term of the right-hand side of \eqref{d1_interm_1}. We recall that $\bar{\mu}(t)= (id,\bar{\alpha}_t) \sharp \bar{m}(t)$. Since $\bar{A}_t^j= \bar{\alpha}_t(\bar{X}_t^j)$, we also have $\mu^N_{\bar{\bm{A}}}(t) = (id, \bar{\alpha}_t)\sharp \bar{m}^{N}_{\bar{\bm{A}}}(t)$.
We deduce from the Lipschitz continuity of $(id,\bar{\alpha}_t)$ and from Lemma \ref{g-sharp-m} that
\begin{align} \nonumber
d_1\left(\mu^N_{\bar{\bm{A}}}(t) ,\bar{\mu}(t)\right) & = d_1\left((id, \bar{\alpha}_t)\sharp \bar{m}^{N}_{\bar{\bm{A}}}(t),(id, \bar{\alpha}_t)\sharp \bar{m}(t)\right)
\\ & \leq C d_1\left( \bar{m}^{N}_{\bar{\bm{A}}}(t), \bar{m}(t)\right). \label{d1-mu-m}
\end{align}
The random variables $\bar{X}_t^j$ are independent and $\mathcal{L}(\bar{X}^j_t) \sim \bar{m}(t)$. Therefore, Theorem \ref{fournier-thm} applies and yields
\begin{equation} \label{E-d1-CN}
\mathbb{E} \left[ d_1\left( \bar{m}^{N}_{\bar{\bm{A}}}(t), \bar{m}(t)\right) \right] \leq CN^{-\tau(d)}.
\end{equation}
Combining \eqref{d1_interm_1}, \eqref{a_1}, \eqref{d1-mu-m}, and \eqref{E-d1-CN}, we obtain
\begin{equation*}
\mathbb{E} \left[ d_1\left(\mu^N_{(A^i,\bar{\bm{A}}^{-i})}(t) ,\bar{\mu}(t)\right) \right] \leq CN^{-\tau(d)}.
\end{equation*}
It is then easy to verify that
\begin{equation*}
\mathbb{E} \left[ d_1\left(m^N_{(A^i,\bar{\bm{A}}^{-i})}(T) ,\bar{m}(T) \right) \right] \leq CN^{-\tau(d)}.
\end{equation*}
Estimate \eqref{d1-estim-i} follows immediately.
\end{proof}

\begin{lmm} \label{abs-rho-F}
There exists $C>0$ such that for any $i\in \mathcal{N}$, for any $t\in \TT$, and for any $A^i \in \bm{\mathcal{A}}^c$, we have
\begin{align*}
\mathbb{E} \left[ \, \left| \Delta F\left(t,X^i_t[A^i],b_{(A^i,\bar{\bm{A}}^{-i})}^{N},\bar{b}\right)\right| \, \right]   \leq CN^{-\tau(d)/2}.
\end{align*}
\end{lmm}

\begin{proof}
Let $N \in \mathbb{N}^\star$ and let $t\in \TT$. By Remark \ref{remark-F-holder}, we have
\begin{equation} \label{X2-bounded}
\mathbb{E} \left[ \, \left| \Delta F\left(t,X^i_t[A^i],b_{(A^i,\bar{\bm{A}}^{-i})}^{N},\bar{b}\right)  \right| \, \right] \leq C \mathbb{E} \left[\left(1+|X^i_t[A^i]|\right)d_1\left(b_{(A^i,\bar{\bm{A}}^{-i})}^{N},\bar{b}\right)^{1/2} \right].
\end{equation}
Since $A^i \in \bm{\mathcal{A}}^c$, by Remark \ref{X0-rem} we have that
$\mathbb{E} \left[|X^i_t[A^i]|^2\right] \leq C$.
We obtain with the Cauchy-Schwarz inequality and Lemma \ref{lemma-dist-d1} that
\begin{align} \nonumber
\mathbb{E} \left[(1+|X^i_t[A^i]|) d_1\left(b_{(A^i,\bar{\bm{A}}^{-i})}^{N},\bar{b}\right)^{1/2} \right] & \leq \mathbb{E} \left[(1+|X^i_t[A^i]|)^2 \right]^{1/2} \mathbb{E} \left[d_1\left(b_{(A^i,\bar{\bm{A}}^{-i})}^{N},\bar{b}\right)\right]^{1/2} \\ & \leq  CN^{-\tau(d)/2}.\label{EX-d1-CN}
\end{align}
Combining \eqref{X2-bounded} and \eqref{EX-d1-CN}, we obtain the announced inequality.
\end{proof}

\begin{lmm} \label{abs-rho-P}
There exists $C>0$ such that for any $i \in \mathcal{N}$, for any $t\in \T$ and for any $A^i \in \bm{\mathcal{A}}^c$, we have
\begin{align} \label{main-estimates-P}
\mathbb{E} \left[ \, \left|\langle A^i_t, \Delta P\left(t,b_{(A^i,\bar{\bm{A}}^{-i})}^{N},\bar{b}\right) \rangle \right| \, \right]  \leq CN^{-\tau(d)/2}.
\end{align}
\end{lmm}

\begin{proof}
Let $i \in \mathcal{N}$, let $t \in \T$, and let $A^i \in \bm{\mathcal{A}}^c$.
By the Cauchy-Schwarz inequality, we have
\begin{equation} \label{abs-alpha-p-p}
\mathbb{E}\left[ \, \left| \langle A^i_t, \Delta P\left(t,b_{(A^i,\bar{\bm{A}}^{-i})}^{N},\bar{b}\right) \rangle \right| \, \right] \leq C \left( \mathbb{E}\left[ \, \left| \Delta P\left(t,b_{(A^i,\bar{\bm{A}}^{-i})}^{N},\bar{b}\right)\right|^2 \, \right] \right)^{1/2}.
\end{equation}
We obtain with Assumptions \ref{B} (iii-iv) and Lemma \ref{lemma-dist-d1} that
\begin{align} \nonumber
\mathbb{E}\left[ \, \left|\Delta P\left(t,b_{(A^i,\bar{\bm{A}}^{-i})}^{N},\bar{b}\right) \right|^2 \, \right]
& \leq 2C \mathbb{E} \left[ \, \left| \Delta P\left(t,b_{(A^i,\bar{\bm{A}}^{-i})}^{N},\bar{b}\right) \right| \, \right]  \\
& \leq C \mathbb{E} \left[ d_1 \left(b_{(A^i,\bar{\bm{A}}^{-i})}^{N},\bar{b} \right) \right]
\leq CN^{-\tau(d)}.\label{p-p-cn}
\end{align}
Combining (\ref{abs-alpha-p-p}) and (\ref{p-p-cn}), we deduce \eqref{main-estimates-P}.
\end{proof}

We finally prove the main result of the section.

\begin{proof}[Proof of Theorem \ref{thm:approx_nash}] \label{proof:main_thm}
Let $i\in \mathcal{N}$.
We first show that for any $A^i \in \bm{\mathcal{A}}^{c}$, the inequality
\begin{equation} \label{J-ineq}
|\mathcal{J}^{i,N}(A^i,\bm{\bar{A}}^{-i})  -  \mathcal{J}^{i}(A^i,\bar{b})| \leq  CN^{-\tau(d)/2}
\end{equation}
holds for some constant $C > 0$ independent of $A^i$. This will imply \eqref{eq:convergence_cost}. For any $A^i \in \bm{\mathcal{A}}^{c}$, we can write
$\mathcal{J}^{i,N}(A^i,\bm{\bar{A}}^{-i})= \rho^i(Y)$
and
$\mathcal{J}^{i}(A^i,\bar{b})= \rho^i(Z)$,
where
\begin{align*}
Y:= \ & \sum_{t=0}^{T-1} \ell\left(t,X_t^i[A^i],A^i_t,b_{(A^i,\bar{\bm{A}}^{-i})}^{N}\right) + F\left(T,X_T^i[A^i],b_{(A^i,\bar{\bm{A}}^{-i})}^{N}\right), \\
Z:= \ & \sum_{t=0}^{T-1} \ell\left(t,X_t^i[A^i],A^i_t,\bar{b}\right) + F\left(T,X_T^i[A^i], \bar{b} \right).
\end{align*}
Applying Lemma \ref{ineq-rho-triangle} with $U = Z$ and $V = Y-Z$ we have
\begin{equation*}
 |\mathcal{J}^{i,N}(A^i,\bm{\bar{A}}^{-i})  -  \mathcal{J}^{i}(A^i,\bar{b})| = | \rho^i( Y) -  \rho^i( Z) | \leq \rho^i( |Y-Z|).
\end{equation*}
In addition, Lemma \ref{ineq-esp-rho-lemma} yields
\begin{equation*}
\rho^i( |Y-Z|) \leq C \mathbb{E} \left[ \, |Y-Z| \, \right] = C \mathbb{E} \left[ \, \left| \Delta \ell\left(X^i[A^i],b_{(A^i,\bar{\bm{A}}^{-i})}^{N},\bar{b} \right) \right| \, \right].
\end{equation*} 
We finally obtain \eqref{J-ineq} with Lemma \ref{abs-rho-F} and Lemma \ref{abs-rho-P}.

Let us fix now $\widehat{A}^i \in \bm{\mathcal{A}}$ such that
\begin{equation} \label{Ji-Ahat-Ji}
\mathcal{J}^{i,N}(\widehat{A}^i, \bar{\bm{A}}^{-i}) \leq \left( \inf_{A^i \in \bm{\mathcal{A}}} \mathcal{J}^{i,N}(A^i, \bar{\bm{A}}^{-i}) \right) + \min \left\{ 1, N^{-\tau(d)/2} \right\}.
\end{equation}
By Lemma \ref{alpha0-AC}, we have $\widehat{A}^i \in \bm{\mathcal{A}}^c$.
Thus inequality \eqref{J-ineq} yields
\begin{align} 
\mathcal{J}^{i}(\widehat{A}^i,\bar{b})
\leq \ & \mathcal{J}^{i,N}(\widehat{A}^i,\bm{\bar{A}}^{-i}) + C N^{-\tau(d)/2} \notag \\
\leq \ & \left( \inf_{A^i \in \bm{\mathcal{A}}} \mathcal{J}^{i,N}(A^i, \bar{\bm{A}}^{-i}) \right) + C N^{-\tau(d)/2}. \label{ineq-2-J}
\end{align}
We apply again inequality \eqref{J-ineq} to $A^i= \bar{A}^i$. Using also the optimality of $\bar{A}^i$ (with respect to $\mathcal{J}^i$), we obtain
\begin{equation} \label{ineq-Ji-Ji}
 \mathcal{J}^{i,N}(\bar{\bm{A}})- C N^{-\tau(d)/2} \leq \mathcal{J}^{i}(\bar{A}^i,\bar{b}) \leq \mathcal{J}^{i}(\widehat{A}^i,\bar{b}).
\end{equation}
Finally, combining \eqref{ineq-2-J} and \eqref{ineq-Ji-Ji} we have
\begin{equation*}
 \mathcal{J}^{i,N}(\bar{\bm{A}}) \leq \left( \inf_{A^i \in \bm{\mathcal{A}}} \mathcal{J}^{i,N}(A^i, \bar{\bm{A}}^{-i}) \right) + C N^{-\tau(d)/2},
\end{equation*}
which shows that $\bar{\bm{A}}$ is an $\varepsilon$-Nash equilibrium with $\varepsilon = CN^{-\tau(d)/2}$.
\end{proof}

\section{Conclusion}

This paper has studied a mean field game model with risk averse agents, and provided a framework under which an equilibrium holds, for a large
class of composite risk measures and congestion terms.
The specific structure of the integral cost of the agents has been exploited in order to rewrite the dynamic programming equations in a functional form (using the Moreau envelope and the proximal operator). In that way, the coupled system could be formulated as an equivalent fixed point equation, yielding the existence of a solution.
Regularity properties have been obtained for risk averse agents. This has allowed to show that an optimal feedback control (for the mean field game) results in an $\varepsilon$-Nash equilibrium for a related dynamic game with $N$ players.
Future work could focus on the uniqueness of the Nash equilibrium with contraction arguments and smallness assumptions on the coupling terms.
In this work, risk averse (with respect to their own noise) agents have been considered; investigating a mean field game model with common noise and risk averse agents would be of particular interest.
Finally, we could investigate variants of our model involving agents driven by nonlinear dynamical systems, nonconvex data functions, or exponential utility cost functions. In such a setting we cannot expect anymore the value function to be convex and thus, a feedback policy cannot be defined in a unique manner. A different notion of equilibrium must then be employed. An appropriate one may rely on the distribution of the controls of the agents at each time, conditioned to their position, as for example in \cite{basar-discrete}, where an existence result is obtained with Kakutani's theorem.

\bibliographystyle{plain}
\bibliography{biblio}

\end{document}